\documentclass[11pt]{elsarticle}
\usepackage{mathrsfs}

 \usepackage{graphics}

\usepackage{amssymb}

\usepackage{amsfonts}
\usepackage{pifont}
\usepackage{amsmath}
\usepackage{latexsym}
\usepackage{bm}
\usepackage{amssymb}
\usepackage{amsthm}
\usepackage{graphics}
\usepackage{epstopdf}
\usepackage{tikz}
\def\bi{\begin{align}}
\def\bin{\begin{align*}}
\numberwithin{equation}{section}
\numberwithin{figure}{section}
\numberwithin{table}{section}

\newtheorem{thm}{Theorem}[section]

\newtheorem{lem}[thm]{Lemma}

\newtheorem{defn}[thm]{Definition}
\newtheorem{rem}[thm]{Remark}
\newtheorem{exa}{Example}[section]

\newcommand{\dif}{\mathrm{d}}

\newcommand{\ba}{\begin{array}}
\newcommand{\ea}{\end{array}}

\begin{document}

\begin{frontmatter}



\title{Continuous-stage Runge-Kutta-Nystr\"{o}m methods}

\author[a,b]{Wensheng Tang\corref{cor1}}
\ead{tangws@lsec.cc.ac.cn}\cortext[cor1]{Corresponding author.}
\address[a]{College of Mathematics and Statistics,\\
    Changsha University of Science and Technology,  Changsha 410114, China}
\address[b]{Hunan Provincial Key Laboratory of \\
    Mathematical Modeling and Analysis in Engineering,  Changsha 410114, China}
\author[]{}

\begin{abstract}

We develop continuous-stage Runge-Kutta-Nystr\"{o}m (csRKN) methods
in this paper. By leading weight function into the formalism of
csRKN methods and modifying the original pattern of continuous-stage
methods, we establish a new and larger framework for csRKN methods
and it enables us to derive more effective RKN-type methods.
Particularly, a variety of classical weighted orthogonal polynomials
can be used in the construction of RKN-type methods. As an important
application, new families of symmetric and symplectic integrators
can be easily acquired in such framework. Numerical experiments have
verified the effectiveness of the new integrators presented in this
paper.

\end{abstract}


\end{frontmatter}


\section{Introduction}
\label{}



The seminal idea of continuous-stage methods was introduced by
Butcher (1972) in \cite{butcher72ato} (see also
\cite{butcher87tna,butcherw96rkm} for a more detailed description),
which suggests a ``continuous" extension of Runge-Kutta (RK) methods
by allowing the number of stages to be infinite so that the discrete
index set $\{1,2,\cdots,s\}$ becomes the interval $[0,1]$.
Unfortunately, this creative idea has been completely ignored in a
very long period of time. Such situation was continued until the
year 2010, Hairer activated the idea by using it to interpret his
energy-preserving collocation methods \cite{hairer10epv} and then an
elegant mathematical formalism for continuous-stage Runge-Kutta
methods was created by him. Since then, there has been a revival of
interest in the study of continuous-stage methods, and some
researchers consciously or unconsciously conduct their studies
closely related with such a subject. The first related work after
Hairer's was given by Tang \& Sun \cite{Tangs12tfe}, stating that
there is an interesting connection between Galerkin variational
methods and continuous-stage methods, and it was shown in
\cite{Tangs12tfe} that energy-preserving methods such as $s$-stage
trapezoidal methods \cite{Iavernarop07sst}, average vector field
methods \cite{quispelm08anc}, and infinite Hamiltonian boundary
value methods \cite{brugnanoit10hbv} (as well as Hairer's
energy-preserving collocation methods \cite{hairer10epv}) can be
unified in the framework of continuous-stage methods. In recent
years, there are a series of papers intensively studying in such
subject
\cite{Tangs12ana,Tangs14cor,Tanglx16cos,Tangz18spc,Tang18ano,
Tang18csr,Tang18csm,Tang18siw,Tang18aef,Tangz18sib}.

So far, the available methods with continuous stage can be grouped
into the following three classes: continuous-stage Runge-Kutta
(csRK) methods \cite{Tangs12ana,Tangs14cor,Tang18ano,Tang18csr},
continuous-stage partitioned Runge-Kutta (csPRK) methods
\cite{Tanglx16cos}, and continuous-stage Runge-Kutta-Nystr\"{o}m
(csRKN) methods \cite{Tangsz18hos,Tangz18spc}. It turns out that
with the idea of continuous-stage methods we can easily construct
many effective integrators of arbitrarily-high order, without
needing to solve the tedious nonlinear algebraic equations (usually
associated with the order conditions) in terms of many unknown
coefficients. Particularly, a crucial technique for constructing
continuous-stage methods with arbitrary order is developed in
\cite{Tangs14cor,Tangsz18hos,Tanglx16cos,Tangz18spc,Tang18csr},
which is mainly based on the orthogonal polynomial expansion.

Compared with standard RK \& RK-like discretizations, the
continuous-stage approaches may provide us a new insight in many
aspects of numerical solution of differential equations, seeing that
the Butcher coefficients (as functions) are assumed to be
``continuous" or ``smooth" which potentially allows us to use some
analytical tools such as Taylor expansion, inner product, limit
operation, orthogonal expansion, differentiation, integration, etc
\cite{liw16ffe,Tangs14cor,Tangsz18hos,Tanglx16cos,Tangz18spc,Tang18csr}.
Owing to this point, sometimes it may lead to surprising
applications. A good case in point is that no RK methods are
energy-preserving for general non-polynomial Hamiltonian systems
\cite{Celledoni09mmoqw}, whereas energy-preserving csRK methods can
be easily constructed
\cite{brugnanoit10hbv,hairer10epv,miyatake14aee,
miyatakeb16aco,quispelm08anc,Tangs12ana,Tangs12tfe,Tangs14cor}.
Another example is given by Tang \& Sun \cite{Tangs12tfe}, which
states that some Galerkin variational methods can be interpreted as
continuous-stage (P)RK methods, but they can not be completely
understood in the classical (P)RK framework.

Over the last few decades, geometric integration for the numerical
solution of differential equations has attracted much attention
(see, for example,
\cite{Benetting94oth,Channels90sio,Feng84ods,Feng95kfc,Fengqq10sga,
hairerlw06gni,lasagni88crk,Leimkuhlerr04shd,ruth83aci,sanz88rkm,
sanzc94nhp,suris88otc,suris89ctg,Vogelaere56moi}), for the reason
that numerical discretization respecting the geometric properties of
the exact flow are very important for long-time integration
\cite{Benetting94oth,hairerlw06gni,Shang99kam,Tang94feo}. In recent
years, continuous-stage methods have found their interesting
applications in geometric integration. For example, symplectic and
multi-symplectic integrators can be derived by using Galerkin
variational approaches, and these integrators can be interpreted and
analyzed in the framework of continuous-stage methods
\cite{Tangs12tfe,Tangsc17dgm,Tang18sio}; some newly-developed
energy-preserving methods can be closely related to continuous-stage
methods
\cite{brugnanoit10hbv,Celledoni09mmoqw,cohenh11lei,hairer10epv,liw16ffe,
miyatake14aee,miyatakeb16aco,quispelm08anc,Tangs12tfe,Tang18epi};
new families of symplectic and symmetric methods can be constructed
by using the idea of continuous-stage methods
\cite{Tangs12ana,Tangs14cor,Tangz18spc,Tang18ano,Tang18csr,Tang18csm,
Tang18siw,Tang18aef,Tangz18sib}; the study of conjugate
symplecticity of energy-preserving methods may be promoted in the
context of continuous-stage methods \cite{hairer10epv,hairerz13oco},
etc. Undoubtedly, other new applications of continuous-stage methods
in geometric integration are actively under development.

More recently, the present author et al.
\cite{Tangz18spc,Tangsz18hos} have developed symplectic RKN-type
integrators by virtue of continuous-stage methods. In this paper, we
are going to enlarge the primitive framework of csRKN methods to a
new one which enables us to treat more complicated cases. For this
sake, by using the similar idea presented in \cite{Tang18aef}, we
will lead weight function into the formalism of csRKN methods and
define the continuous-stage methods in a general interval $I$
(finite or infinite). By doing this, a variety of classical weighted
orthogonal polynomials can be used in the construction of RKN-type
methods. As an important application, new symmetric and symplectic
integrators can be easily derived in this new framework.

This paper will be organized as follows. In Section 2, we introduce
the new definition of csRKN methods for solving second-order
differential equations. This is followed by Section 3, where the
order theory by using simplifying assumptions will be given. Section
4 is devoted to present our approach for deriving symmetric and
symplectic integrators accompanied with some examples. We exhibit
our numerical results in Section 5. At last, we end our paper in
Section 6 with some concluding remarks.

\section{Continuous-stage Runge-Kutta-Nystr\"{o}m methods}

We are concerned with the initial value problem governed by a
second-order system
\begin{align}\label{eq:second}
q''=f(t, q),\;\;q(t_0)=q_0,\;\;q'(t_0)=q'_0,
\end{align}
where $f:\mathbb{R}\times\mathbb{R}^{d}\rightarrow\mathbb{R}^{d}$ is
assumed to be a smooth vector-valued function.
\begin{defn}\label{weight_func}
A non-negative function $w(x)$ is called a \emph{weight function} on
the interval $I$, if it satisfies the following two conditions:
\begin{itemize}
 \item[(a)] The $k$-th moment $\int_I x^k w(x)\,\dif x, \;k\in\mathbb{N}$ exists;
 \item[(b)] For any non-negative function $u(x)$, $\int_Iu(x)w(x)\,\dif
 x=0$ implies $u(x)\equiv0$.
 \end{itemize}
\end{defn}

Based on the notion of weight function, we introduce the following
definition of continuous-stage Runge-Kutta-Nystr\"{o}m methods which
is an extended version of that given in
\cite{Tangsz18hos,Tangz18spc}.

\begin{defn}\label{csRKN:def}
Let $w(x)$ be a weight function defined on $I$ (finite or infinite),
$\bar{A}_{\tau, \sigma}$ be a function of variables $\tau, \sigma\in
I$ and $\bar{B}_\tau,\;B_\tau,\;C_\tau$ be functions of $\tau\in I$.
The continuous-stage Runge-Kutta-Nystr\"{o}m (csRKN) method for
solving \eqref{eq:second} is given by
\begin{subequations}
    \begin{alignat}{2}
    \label{eq:csrkn1}
&Q_\tau=q_0 +hC_\tau q'_0 +h^2\int_I \bar{A}_{\tau, \sigma}
w(\sigma)
f(t_0+C_\sigma h, Q_\sigma) \dif \sigma, \;\;\tau \in I,\\
\label{eq:csrkn2} &q_{1}=q_0+ h q'_0+h^2 \int_I \bar{B}_\tau
w(\tau)f(t_0+C_\tau h, Q_\tau) \dif\tau, \\
\label{eq:csrkn3} &q'_1
= q'_0 +h\int_I B_\tau w(\tau) f(t_0+C_\tau h, Q_\tau) \dif \tau,
    \end{alignat}
\end{subequations}
which can be characterized by the following Butcher tableau
\[\ba{c|c} C_\tau & \bar{A}_{\tau,\sigma}w(\sigma)\\[4pt]
\hline & \bar{B}_\tau w(\tau)\\ \hline\\[-15pt] & B_\tau w(\tau)\ea\]
\end{defn}

\begin{rem}
For the case when $I$ is an infinite interval, we assume that the
improper integrals of (\ref{eq:csrkn1}-\ref{eq:csrkn3}) satisfy some
conditions (in terms of uniform convergence) such that
differentiation under the integral sign with respect to parameter
$h$ (step size) is legal.
\end{rem}

\begin{rem}
If we let $I=[0,1]$ and $w(x)=1$, then it results in the methods
developed in \cite{Tangsz18hos,Tangz18spc}. However, remark that the
primitive framework of csRKN methods given in
\cite{Tangsz18hos,Tangz18spc} can not be applicable for more
complicated cases, e.g., the case for weighting on a infinite
interval $(-\infty,+\infty)$ or any other general interval $I$.
\end{rem}

\section{Discussions on the order theory}

\begin{defn}\cite{hairernw93sod}
A csRKN method is called order $p$, if for all sufficiently regular
problem \eqref{eq:second}, as $h\rightarrow0$, its \emph{local
error} satisfies
\begin{equation*}
q(t_0+h)-q_1=\mathcal{O}(h^{p+1}),\quad
q'(t_0+h)-q'_1=\mathcal{O}(h^{p+1}).
\end{equation*}
\end{defn}

\subsection{The order of csRKN methods}

Following the idea of classical cases
\cite{hairernw93sod,hairerlw06gni}, we propose the following
simplifying assumptions\footnote{It should be noticed that in
$\mathcal{DN}(\zeta)$ we have removed ``$w(\sigma)$" from both sides
of the formula.}
\begin{equation*}\label{csRKN-simpl-assump}
\begin{split}
&\mathcal{B}(\xi):\quad \int_IB_\tau w(\tau)C_\tau^{\kappa-1}\,\dif
\tau=\frac{1}{\kappa},\;\; 1\leq\kappa\leq\xi,\\
&\mathcal{CN}(\eta):\quad
\int_I\bar{A}_{\tau,\,\sigma}w(\sigma)C_\sigma^{\kappa-1}\,\dif
\sigma=\frac{C_\tau^{\kappa+1}}{\kappa(\kappa+1)},\;\;1\leq\kappa\leq\eta-1,\\
&\mathcal{DN}(\zeta):\quad \int_IB_\tau w(\tau) C_\tau^{\kappa-1}
\bar{A}_{\tau,\,\sigma}\,\dif \tau=\frac{B_\sigma
C_\sigma^{\kappa+1}}{\kappa(\kappa+1)}-\frac{B_\sigma
C_\sigma}{\kappa} +\frac{B_\sigma}{\kappa+1},\;\;
1\leq\kappa\leq\zeta-1,
\end{split}
\end{equation*}
where $ \tau,\,\sigma\in I$.
\begin{thm}\label{ord_csRKN}
If the csRKN method (\ref{eq:csrkn1}-\ref{eq:csrkn3}) with its
coefficients satisfying the simplifying assumptions
$\mathcal{B}(p),\,\mathcal{CN}(\eta),\,\mathcal{DN}(\zeta)$, and if
$\bar{B}_\tau=B_\tau(1-C_\tau)$ is always fulfilled, then the method
is of order at least $\min\{p,\,2\eta+2,\eta+\zeta\}$.
\end{thm}
\begin{proof}
This is a straightforward result of Theorem 3.3 in
\cite{Tangsz18hos}.
\end{proof}

In what follows, we will use the hypothesis $C_\tau=\tau$ (and thus
$\bar{B}_\tau=B_\tau(1-\tau)$) throughout this paper. Let us
establish a lemma in the first place.
\begin{lem}\label{lem_assum}
With the hypothesis $C_\tau=\tau$, the simplifying assumptions
$\mathcal{B}(\xi), \mathcal{CN}(\eta)$ and $\mathcal{DN}(\zeta)$ are
equivalent to, respectively, {\small\begin{align}\label{eq:cd1}
&\mathcal{B}(\xi):\; \int_IB_\tau w(\tau) \phi(\tau)\,\dif
\tau=\int_0^1\phi(x)\,\dif x,\;\;  \forall\, \phi\;
\text{with}\;\emph{deg}(\phi)\leq\xi-1,\\\label{eq:cd2}
&\mathcal{CN}(\eta):\; \int_I\bar{A}_{\tau,\,\sigma} w(\sigma)
\phi(\sigma)\,\dif \sigma=\int_0^\tau \int_0^\alpha\phi(x)\,\dif
x\,\dif\alpha,\;\; \forall\,  \phi\; \text{with}\;
\emph{deg}(\phi)\leq\eta-2,\\\label{eq:cd3} &\mathcal{DN}(\zeta):\;
\int_IB_\tau \bar{A}_{\tau,\,\sigma}w(\tau)\phi(\tau)\,\dif
\tau=B_\sigma\Big(\int_0^\sigma\int_1^\alpha\phi(x)\,\dif
x\,\dif\alpha+\int_0^1x\phi(x)\,\dif x\Big),\;\;\forall\, \phi\;
\text{with}\; \emph{deg}(\phi)\leq\zeta-2,
\end{align}}
where $\emph{deg}(\phi)$ stands for the degree of polynomial
function $\phi$.
\end{lem}
\begin{proof}
With the hypothesis $C_\tau=\tau$, we can rewrite $\mathcal{B}(\xi),
\mathcal{CN}(\eta)$ and $\mathcal{DN}(\zeta)$ as
{\small\begin{equation*}
\begin{split}
&\mathcal{B}(\xi):\;\; \int_IB_\tau w(\tau)\tau^{\kappa-1}\,\dif
\tau=\int_0^1x^{\kappa-1}\,\dif
x,\;\; 1\leq\kappa\leq\xi,\\
&\mathcal{CN}(\eta):\;\;
\int_I\bar{A}_{\tau,\,\sigma}w(\sigma)\sigma^{\kappa-1}\,\dif
\sigma=\int_0^\tau \int_0^\alpha x^{\kappa-1}\,\dif
x\,\dif\alpha,\;\;1\leq\kappa\leq\eta-1,\\
&\mathcal{DN}(\zeta):\;\; \int_IB_\tau w(\tau) \tau^{\kappa-1}
\bar{A}_{\tau,\,\sigma}\,\dif
\tau=B_\sigma\Big(\int_0^\sigma\int_1^\alpha x^{\kappa-1}\,\dif
x\,\dif\alpha+\int_0^1x\cdot x^{\kappa-1}\,\dif x\Big),\;\;
1\leq\kappa\leq\zeta-1,
\end{split}
\end{equation*}}
Therefore, these formulae are satisfied for all monomials like
$x^\iota$ with degree $\iota$ no lager than $\xi-1, \eta-2$ and
$\zeta-2$ respectively. Consequently, the final result follows from
the fact that any polynomial function $\phi$ can be expressed as a
linear combination of monomials.
\end{proof}
It is known that for a given weight function $w(x)$, there exists a
sequence of orthogonal polynomials in the \emph{weighted function
space} (Hilbert space) \cite{Szeg85op}
\begin{equation*}
L^2_w(I)=\{u \text{ is measurable on}\, I:\;\int_I|u(x)|^2w(x)\,\dif
x<+\infty\}
\end{equation*}
which is linked with the inner product
\begin{equation}\label{w_ip}
\big<u,v\big>_w=\int_Iu(x)v(x)w(x)\,\dif x.
\end{equation}

To proceed with our discussions, we denote a sequence of weighted
orthogonal polynomials by $\{P_n(x)\}_{n=0}^\infty$, which consists
of a complete set in the Hilbert space $L^2_w(I)$. It is known that
$P_n(x)$ has exactly $n$ real simple zeros in the interval $I$. For
convenience, in what follows we always assume the orthogonal
polynomials are normalized, i.e., satisfying
\begin{equation*}
\big<P_i,P_j\big>_w=\delta_{ij},\;\;i, j=0,1,2,\cdots.
\end{equation*}

\begin{thm}\label{ordcon_var}
Let $C_\tau=\tau$ and suppose\footnote{The notation
$\bar{A}_{\ast,\,\sigma}$ stands for the one-variable function in
terms of $\sigma$, and $\bar{A}_{\tau,\,\ast}$ can be understood
likewise.}
$B_{\tau},\,\,\bar{A}_{\ast,\,\sigma},\,\,(B_\tau\,\bar{A}_{\tau,\,\ast})\in
L^2_w(I)$, then we have
\begin{itemize}
\item[\emph{(a)}] $\mathcal{B}(\xi)$ holds $\Longleftrightarrow$ $B_\tau$ has
the following form in terms of the normalized orthogonal polynomials
in $L^2_w(I)$:
\begin{equation}\label{Bt}
B_\tau=\sum\limits_{j=0}^{\xi-1}\int_0^1P_j(x)\,\dif x
P_j(\tau)+\sum\limits_{j\geq\xi}\lambda_j P_j(\tau),
\end{equation}
where $\lambda_j$ are any real parameters;
\item[\emph{(b)}] $\mathcal{CN}(\eta)$ holds $\Longleftrightarrow$ $\bar{A}_{\tau,\,\sigma}$ has
the following form in terms of the normalized orthogonal polynomials
in $L^2_w(I)$:
\begin{equation}\label{Ats}
\bar{A}_{\tau,\,\sigma}=\sum\limits_{j=0}^{\eta-2}\int_0^\tau
\int_0^\alpha P_j(x)\,\dif x\,\dif\alpha
P_j(\sigma)+\sum\limits_{j\geq\eta-1}\phi_j(\tau) P_j(\sigma),
\end{equation}
where $\phi_j(\tau)$ are any $L^2_w$-integrable  real functions;
\item[\emph{(c)}] $\mathcal{DN}(\zeta)$ holds $\Longleftrightarrow$ $B_\tau \bar{A}_{\tau,\,\sigma}$ has
the following form in terms of the normalized orthogonal polynomials
in $L^2_w(I)$:
\begin{equation}\label{BtAts}
B_\tau\,\bar{A}_{\tau,\,\sigma}=\sum\limits_{j=0}^{\zeta-2}B_\sigma\Big(\int_0^\sigma\int_1^\alpha
P_j(x)\,\dif x\,\dif\alpha+\int_0^1xP_j(x)\,\dif x\Big)
P_j(\tau)+\sum\limits_{j\geq\zeta-1}\psi_j(\sigma) P_j(\tau),
\end{equation}
where $\psi_j(\sigma)$ are any $L^2_w$-integrable real functions.
\end{itemize}
\end{thm}
\begin{proof}
This theorem can be proved in the same manner as Theorem 2.3 of
\cite{Tang18csr}. For part $(a)$, consider the following orthogonal
polynomial expansion in $L^2_w(I)$
\begin{equation*}
B_\tau=\sum\limits_{j\geq0}\lambda_j P_j(\tau),\;\;\lambda_j\in
\mathbb{R},
\end{equation*}
and substitute the formula above into \eqref{eq:cd1} (with $\phi$
replaced by $P_j$) in Lemma \ref{lem_assum}, then it follows
\begin{equation*}
\lambda_j=\int_0^1P_j(x)\,\dif x,\;\;j=0,\cdots,\xi-1,
\end{equation*}
which gives \eqref{Bt}. For part $(b)$ and $(c)$, consider the
following orthogonal expansions of $\bar{A}_{\tau,\,\sigma}$ with
respect to $\sigma$ and $B_\tau\,\bar{A}_{\tau,\,\sigma}$ with
respect to $\tau$ in $L^2_w(I)$, respectively,
\begin{equation*}
\bar{A}_{\tau,\,\sigma}=\sum\limits_{j\geq0}\phi_j(\tau)
P_j(\sigma),\;\;\phi_j(\tau)\in L^2_w(I),
\end{equation*}
\begin{equation*}
B_\tau\,\bar{A}_{\tau,\,\sigma}=\sum\limits_{j\geq0}\psi_j(\sigma)
P_j(\tau),\;\;\psi_j(\sigma)\in L^2_w(I),
\end{equation*}
and then substitute them into \eqref{eq:cd2} and \eqref{eq:cd3},
which then leads to the final results.
\end{proof}

\begin{rem}\label{rem:csRKN_trunc}
For the sake of obtaining a practical csRKN method, we have to
define a finite form for $B_\tau$ and $\bar{A}_{\tau,\,\sigma}$. A
natural and simple way is to truncate the series \eqref{Bt} and
\eqref{Ats}. As a consequence, $B_\tau$ and
$\bar{A}_{\tau,\,\sigma}$ become polynomial functions.
\end{rem}

\subsection{The order of RKN methods by using quadrature formulas}

In the practical implementation, generally we have to approximate
the integrals of the csRKN method by numerical quadrature formulas.
For this sake, we introduce the following $s$-point \emph{weighted
interpolatory quadrature formula}
\begin{equation}\label{wquad}
\int_I\Phi(\tau)w(\tau)\,\dif \tau\approx\sum\limits_{i=1}^sb_i
\Phi(c_i),\;\;c_i\in I,
\end{equation}
where
\begin{equation*}
b_i=\int_I\ell_i(\tau)w(\tau)\,\dif
\tau,\;\;\ell_i(\tau)=\prod\limits_{j=1,j\neq
i}^s\frac{\tau-c_j}{c_i-c_j},\;\;i=1,\cdots,s.
\end{equation*}

After applying the quadrature formula to
\eqref{eq:csrkn1}-\eqref{eq:csrkn3}, it gives rise to an $s$-stage
RKN method
\begin{subequations}
    \begin{alignat}{2}
    \label{eq:rkn1}
&Q_i=q_0 +hC_i q'_0+h^2\sum\limits_{j=1}^{s}b_j\bar{A}_{ij}f(t_0+C_jh,\,Q_j), \quad i=1,\cdots, s, \\
    \label{eq:rkn2}
&q_{1}=q_0+ h q'_0+h^2\sum\limits_{i=1}^{s}b_i\bar{B}_if(t_0+C_ih,\,Q_i), \\
    \label{eq:rkn3}
&q'_{1}= q'_0+h\sum\limits_{i=1}^{s}b_iB_if(t_0+C_ih,\,Q_i),
    \end{alignat}
\end{subequations}
where $Q_i:=Q_{c_i}, \bar{A}_{ij}:=\bar{A}_{c_i, c_j},
\bar{B}_i:=\bar{B}_{c_i}, B_i:=B_{c_i}, C_i:=C_{c_i}=c_i$ (recall
that $C_\tau=\tau$), which can be characterized by
\begin{equation}\label{RKN:qua}
\ba{c|ccc} c_1 & b_1\bar{A}_{11}
& \cdots & b_s\bar{A}_{1s}\\[2pt]
\vdots &\vdots &\vdots\\[2pt]
c_s & b_1\bar{A}_{s1} &
\cdots & b_s\bar{A}_{ss}\\[2pt]
\hline & b_1\bar{B}_{1}  & \cdots & b_s\bar{B}_{s}\\[2pt]
\hline & b_1B_{1}  & \cdots & b_sB_{s}\ea
\end{equation}
In order to analyze the order of the RKN method \eqref{RKN:qua}, we
propose the following result which is closely related with Remark
\ref{rem:csRKN_trunc}.

\begin{thm}\label{qua:csRKN}
Assume the underlying quadrature formula \eqref{wquad} is of order
$p$, and $\bar{A}_{\tau,\,\sigma}$ is of degree $\pi_A^\tau$ with
respect to $\tau$ and of degree $\pi_A^{\sigma}$ with respect to
$\sigma$, and $B_{\tau}$ is of degree $\pi_B^\tau$. If we assume
$C_\tau=\tau,\,\bar{B}_\tau=B_\tau(1-\tau)$, and all the simplifying
assumptions $\mathcal{B}(\xi)$, $\mathcal{CN}(\eta)$,
$\mathcal{DN}(\zeta)$ are fulfilled, then the RKN method
\eqref{RKN:qua} is at least of order
$$\min\{\rho, \,2\alpha+2, \,\alpha+\beta\},$$
where $\rho=\min\{\xi,\,p-\pi_B^\tau\}$,
$\alpha=\min\{\eta,\,p-\pi_A^{\sigma}+1\}$ and $\beta=\min\{\zeta,\,
p-\pi_A^\tau-\pi_B^\tau+1\}$.
\end{thm}
\begin{proof}
Since the quadrature formula \eqref{wquad} holds for any polynomial
$\Phi(x)$ of degree up to $p-1$, by using it to compute the
integrals of $\mathcal{B}(\xi)$, $\mathcal{CN}(\eta)$,
$\mathcal{DN}(\zeta)$ it gives
\begin{equation*}
\begin{split}
&\sum_{i=1}^s(b_iB_i)c_i^{\kappa-1}=\frac{1}{\kappa},\;\kappa=1,\cdots,\rho,\\
&\sum_{j=1}^s(b_{j}\bar{A}_{ij})c_j^{\kappa-1}=\frac{c_i^{\kappa+1}}{\kappa(\kappa+1)},
\;i=1,\cdots,s,\;\kappa=1,\cdots,\alpha-1,\\
&\sum_{i=1}^s(b_iB_i)c_i^{\kappa-1}(b_{j}\bar{A}_{ij})=
\frac{(b_jB_j)c_j^{\kappa+1}}{\kappa(\kappa+1)}-\frac{(b_jB_j)c_j}{\kappa}+\frac{b_jB_j}{\kappa+1},
\;j=1,\cdots,s,\;\kappa=1,\cdots,\beta-1.
\end{split}
\end{equation*}
where $\rho=\min\{\xi,\,p-\pi_B^\tau\}$,
$\alpha=\min\{\eta,\,p-\pi_A^{\sigma}+1\}$ and $\beta=\min\{\zeta,\,
p-\pi_A^\tau-\pi_B^\tau+1\}$.  These formulas imply that the RKN
method with coefficients given by \eqref{RKN:qua} satisfies the
classical simplifying assumptions $B(\rho)$, ${CN}(\alpha)$ and
${DN}(\beta)$ (see \cite{hairerlw06gni}), and it is observed that we
also have $b_i\bar{B}_i=b_iB_i(1-c_i)$ for each $i=1,\ldots, s$.
Consequently, it gives rise to the order of the method by the
classical result \cite{hairerlw06gni,hairernw93sod}.
\end{proof}

\section{Geometric integration by csRKN methods}

In this section, we discuss the geometric integration by csRKN
methods. As pointed out in \cite{hairerlw06gni}, symplectic
integrators for Hamiltonian systems and symmetric integrators for
reversible systems play a central role in the geometric integration
of differential equations, for the reason that they possess
excellent numerical behaviors in long-time integration. So far,
there are many literatures concentrating on the theoretical analysis
and empirical study of these integrators, see
\cite{Benetting94oth,Feng84ods,Feng95kfc,Fengqq10sga,hairerlw06gni,
Leimkuhlerr04shd,sanzc94nhp} and references therein.

\subsection{Symplectic integrators}

A very important subclass of dynamical systems in classical and
non-classical mechanics are the so-called Hamiltonian systems
\cite{Arnold89mmo}, which read
\begin{equation}\label{Hs}
z'=J^{-1}\nabla_{z}H(z),\;\;z(t_0)=z_0\in\mathbb{R}^{2d},\;\;
z=\begin{pmatrix}
  p \\
  q \\
\end{pmatrix},\;\;
J=\begin{pmatrix}
0 & I_{d\times d} \\
-I_{d\times d} & 0 \\
\end{pmatrix},
\end{equation}
where $q\in\mathbb{R}^{d}$ represents the position coordinates,
$p\in\mathbb{R}^{d}$ the momentum coordinates, and $H$ the
Hamiltonian function (generally represents the total energy). Such
system possesses a symplectic structure (a characteristic property
of the system \cite{hairerlw06gni}), which means the phase flow
$\varphi_t$ satisfies \cite{Arnold89mmo}
\begin{equation*}
\dif\varphi_t(z_0)\wedge J\dif\varphi_t(z_0)=\dif z_0 \wedge J\dif
z_0,\quad \forall\,z_0\in D,
\end{equation*}
where $\wedge$ represents the wedge product, and $D$ is an open
subset in the phase space. For the sake of respecting such geometric
structure in numerical discretization, symplectic integrators are
suggested by some earlier scientists (see
\cite{Feng84ods,ruth83aci,Vogelaere56moi} and references therein),
the definition of which can be stated as follows.

\begin{defn}
A one-step method $\phi_h:
z_0=(p_0,\,q_0)\mapsto(p_{1},\,q_{1})=z_1$ is called symplectic if
and only if
$$\dif\phi_h(z_0)\wedge J\dif\phi_h(z_0)=\dif z_0 \wedge J\dif z_0,\quad \forall\,z_0\in D,$$
whenever the method is applied to a smooth Hamiltonian system.
\end{defn}

A class of Hamiltonian systems frequently encountered in practice is
the following
\begin{equation}\label{eq:first}
p'=-\nabla_q V(q),\;\; q'=Mp,
\end{equation}
with the Hamiltonian
$$H(z)=\frac{1}{2}p^TMp+V(q),$$ where $M$ is a constant symmetric
matrix, and $V(q)$ is a scalar function. This equations can also be
rewritten as a second-order system
\begin{equation}\label{eq:Hs}
q''=-M\nabla_q V(q).
\end{equation}

By using the notations $f(q)=-M\nabla_q V(q)$ and $g(q)=-\nabla_q
V(q)$, we propose the following csRKN method for solving
\eqref{eq:Hs}
\begin{subequations}\label{csRKN:Hs}
\begin{alignat}{2}
    \label{Heq:csrkn1}
&Q_\tau=q_0 +hC_\tau Mp_0 +h^2\int_I \bar{A}_{\tau, \sigma}w(\sigma)
f(Q_\sigma) \dif \sigma, \;\;\tau \in I, \\
    \label{Heq:csrkn2}
&q_{1}=q_0+ h Mp_0+h^2 \int_I \bar{B}_\tau w(\tau)  f(Q_\tau) \dif\tau, \\
    \label{Heq:csrkn3}
&p_1 = p_0 +h\int_I B_\tau w(\tau) g(Q_\tau) \dif\tau.
\end{alignat}
\end{subequations}
Remark that here we have removed the constant matrix $M$ from both
sides of \eqref{Heq:csrkn3} which will not affect the order of the
method. The following theorems have extended the corresponding
results previously presented in \cite{Tangsz18hos}.
\begin{thm}\label{symp_cond_ori}
If the coefficients of a csRKN method
(\ref{Heq:csrkn1}-\ref{Heq:csrkn3}) satisfy
\begin{subequations}
\begin{alignat}{2}
\label{sym_cond_orig01}
\bar{B}_\tau&=B_\tau(1-C_\tau),\quad\tau\in I,\\
\label{sym_cond_orig02}
B_\tau(\bar{B}_\sigma-\bar{A}_{\tau,\sigma})&=B_\sigma(\bar{B}_\tau
-\bar{A}_{\sigma,\tau}),\quad\tau,\sigma\in I,
\end{alignat}
\end{subequations}
then the method is symplectic for solving the system \eqref{eq:Hs}.
\end{thm}
\begin{proof}
The proof is very the same as that of Theorem 4.2 in
\cite{Tangsz18hos} with the range of integration replaced by a
general interval $I$.
\end{proof}
\begin{rem}
Theorem \ref{symp_cond_ori} implies that the symplecticity of the
csRKN methods is independent of its weight function.
\end{rem}
\begin{thm}\label{constr_symcsRKN}
Suppose that $C_\tau=\tau$ and $\bar{A}_{\tau,\sigma}/B_\sigma\in
L_w^2(I\times I)$, then the symplectic condition given in Theorem
\ref{symp_cond_ori} is equivalent to the fact that $\bar{B}_\tau$
and $\bar{A}_{\tau,\sigma}$ have the following form in terms of the
normalized orthogonal polynomials $P_n(x)$ in $L_w^2(I)$
\begin{equation}\label{sym_cond}
\begin{split}
\bar{B}_\tau&=B_\tau(1-\tau),\quad\tau\in I,\\
\bar{A}_{\tau,\sigma}&=B_\sigma\Big(\alpha_{(0,0)}+\alpha_{(0,1)}P_1(\sigma)
+\alpha_{(1,0)}P_1(\tau)+\sum\limits_{i+j>1}\alpha_{(i,j)}
P_i(\tau)P_j(\sigma)\Big),\quad\tau,\sigma\in I,
\end{split}
\end{equation}
where $\alpha_{(0,0)}$ is an arbitrary real number,
$\alpha_{(0,1)}-\alpha_{(1,0)}=-\big<x,\,P_1(x)\big>_w$ (see
\eqref{w_ip}), and the parameters $\alpha_{(i,j)}$ are symmetric,
i.e., $\alpha_{(i,j)}=\alpha_{(j,i)}$ for $\forall\,i+j>1$.
\end{thm}
\begin{proof}
On account of $C_\tau=\tau$, we have
$$\bar{B}_\tau=B_\tau(1-\tau),$$
inserting it into \eqref{sym_cond_orig02}, then it yields
\begin{equation}\label{reduced_symp}
B_\tau\bar{A}_{\tau,\,\sigma}-B_\sigma\bar{A}_{\sigma,\,\tau}=B_\tau
B_\sigma(\tau-\sigma),
\end{equation}
which leads to
\begin{equation}\label{eq:AB0}
\frac{\bar{A}_{\tau,\,\sigma}}{B_\sigma}-\frac{\bar{A}_{\sigma,\,\tau}}{B_\tau}=
\tau-\sigma.
\end{equation}
Here we assume $B_\tau\neq0$, otherwise the csRKN method will be not
practical for possessing no order accuracy. With the help of
$\tau=\sum^1_{i=0}\big<x,\,P_i(x)\big>_w P_i(\tau)$ and notice that
$P_0(\tau)=P_0(\sigma)=constant$, \eqref{eq:AB0} becomes
\begin{equation}\label{eq:AB}
\begin{split}
\frac{\bar{A}_{\tau,\,\sigma}}{B_\sigma}-\frac{\bar{A}_{\sigma,\,\tau}}{B_\tau}&=
\sum^1_{i=0}\big<x,\,P_i(x)\big>_w
P_i(\tau)-\sum^1_{i=0}\big<x,\,P_i(x)\big>_w P_i(\sigma),\\
&=\big<x,\,P_1(x)\big>_w \big(P_1(\tau)-P_1(\sigma)\big).
\end{split}
\end{equation}

Next, consider the expansion of $\bar{A}_{\tau,\,\sigma}/B_\sigma$
along the normalized orthogonal basis
$\left\{P_i(\tau)P_j(\sigma)\right\}_{i,j=0}^\infty$ of
$L^2_w(I\times I)$
\begin{equation}\label{AdivB}
\bar{A}_{\tau,\,\sigma}/B_\sigma=\alpha_{(0,0)}+\alpha_{(0,1)}P_1(\sigma)
+\alpha_{(1,0)}P_1(\tau)+\sum\limits_{i+j>1}\alpha_{(i,j)}
P_i(\tau)P_j(\sigma),\quad\alpha_{(i,j)}\in\mathbb{R}.
\end{equation}
By exchanging $\tau$ and $\sigma$ it gives
\begin{equation*}
\bar{A}_{\sigma,\,\tau}/B_\tau=\alpha_{(0,0)}+\alpha_{(0,1)}P_1(\tau)
+\alpha_{(1,0)}P_1(\sigma)+\sum\limits_{i+j>1}\alpha_{(j,i)}
P_j(\sigma)P_i(\tau),
\end{equation*}
where we have interchanged the indexes $i$ and $j$. By substituting
the above two expressions into \eqref{eq:AB}, it yields
\begin{equation}\label{coef_Ats}
\alpha_{(0,0)}\in\mathbb{R},\;\,\alpha_{(0,1)}-\alpha_{(1,0)}=-\big<x,\,P_1(x)\big>_w,
\;\,\alpha_{(i,j)}=\alpha_{(j,i)},\;\forall\,i+j>1,
\end{equation}
which completes the proof by using \eqref{AdivB}.
\end{proof}

\begin{thm}\label{sym_quad}
If the coefficients of a csRKN method
(\ref{Heq:csrkn1}-\ref{Heq:csrkn3}) satisfies the symplectic
conditions (\ref{sym_cond_orig01}-\ref{sym_cond_orig02}), then the
RKN method \eqref{RKN:qua} derived by using the quadrature formula
\eqref{wquad} is always symplectic.
\end{thm}
\begin{proof}
Please refer to Theorem 4.1 of \cite{Tangz18spc} for a similar
proof.
\end{proof}

\begin{thm}\label{sym_design}
With the hypothesis $C_\tau=\tau$, for a symplectic csRKN method
with coefficients satisfying
(\ref{sym_cond_orig01}-\ref{sym_cond_orig02}), we have the following
statements:
\begin{itemize}
\item[\emph{(a)}] $\mathcal{B}(\xi)$ and $\mathcal{CN}(\eta)$
$\Longrightarrow$ $\mathcal{DN}(\zeta)$, where
$\zeta=\min\{\xi,\,\eta\}$;
\item[\emph{(b)}] $\mathcal{B}(\xi)$ and $\mathcal{DN}(\zeta)$
$\Longrightarrow$ $\mathcal{CN}(\eta)$, where
$\eta=\min\{\xi,\,\zeta\}$.
\end{itemize}
\end{thm}
\begin{proof}
Here we only provide the proof of (a), as (b) can be proved in a
similar manner. From the proof of Theorem \ref{constr_symcsRKN}, we
have get the formula \eqref{reduced_symp} using the hypothesis
$C_\tau=\tau$. Based on this, by multiplying $\sigma^{\kappa-1}$
from both sides of \eqref{reduced_symp} and taking integral it gives
\begin{equation}\label{symp_assum}
B_\tau\int_I\bar{A}_{\tau,\sigma}\sigma^{\kappa-1}\,\dif
\sigma-\int_IB_\sigma\sigma^{\kappa-1} \bar{A}_{\sigma,\tau}\,\dif
\sigma= B_\tau \int_IB_\sigma\sigma^{\kappa-1}(\tau-\sigma)\,\dif
\sigma,\;\;\; \kappa=1,2,\cdots,\zeta-1.
\end{equation}
Now let $\zeta=\min\{\xi,\,\eta\}$, and then $\mathcal{B}(\zeta)$
and $\mathcal{CN}(\zeta)$ can be used for calculating the integrals
of \eqref{symp_assum}. As a result, we have
\begin{equation*}
B_\tau\frac{\tau^{\kappa+1}}{\kappa(\kappa+1)}-\int_IB_\sigma\sigma^{\kappa-1}
\bar{A}_{\sigma,\tau}\,\dif \sigma= \frac{B_\tau
\tau}{\kappa}-\frac{B_\tau}{\kappa+1},\;\;\;
\kappa=1,2,\cdots,\zeta-1.
\end{equation*}
Recall that $C_\tau=\tau$, it gives rise to
\begin{equation*}
\int_IB_\sigma\sigma^{\kappa-1} \bar{A}_{\sigma,\tau}\,\dif \sigma=
\frac{B_\tau C_\tau^{\kappa+1}}{\kappa(\kappa+1)}-\frac{B_\tau
C_\tau}{\kappa} +\frac{B_\tau}{\kappa+1},\;\;\;
\kappa=1,2,\cdots,\zeta-1.
\end{equation*}
Finally, by exchanging $\tau\leftrightarrow\sigma$ in the formula
above, it gives $\mathcal{DN}(\zeta)$ with
$\zeta=\min\{\xi,\,\eta\}$.
\end{proof}
\begin{rem}
A counterpart result for classical symplectic RKN methods can be
similarly obtained.
\end{rem}

On the basis of these preliminaries, following the same idea of
\cite{Tang18siw}, we introduce an operational \emph{procedure} for
deriving symplectic RKN-type integrators:\\
\textbf{Step 1.} Let $C_\tau=\tau,\,\bar{B}_\tau=B_\tau(1-C_\tau)$
and make an ansatz for $B_\tau$  by using \eqref{Bt} so as to
satisfy $B(\xi)$. Note that a finite number of parameters, say
$\lambda_\iota$, could be kept as free parameters;\\
\textbf{Step 2.} Suppose $\bar{A}_{\tau,\,\sigma}$ is in the form
(by Theorem \ref{constr_symcsRKN}, a truncation is needed)
\begin{equation}\label{Ats_proc}
\bar{A}_{\tau,\sigma}=B_\sigma\Big(\alpha_{(0,0)}+\alpha_{(0,1)}P_1(\sigma)
+\alpha_{(1,0)}P_1(\tau)+\sum\limits_{i+j>1}\alpha_{(i,j)}
P_i(\tau)P_j(\sigma)\Big),
\end{equation}
where the parameters $\alpha_{(i,j)}$ satisfy \eqref{coef_Ats}, and
then substitute $\bar{A}_{\tau,\,\sigma}$ into\footnote{An
alternative technique is to consider using $\mathcal{DN}(\zeta)$.}
$\mathcal{CN}(\eta)$ (usually let $\eta<\xi$) for determining
$\alpha_{(i,j)}$:
\begin{equation*}
\int_I\bar{A}_{\tau,\,\sigma}w(\sigma)\phi_k(\sigma)\,\dif
\sigma=\int_0^\tau \int_0^\alpha\phi_k(x)\,\dif
x\,\dif\alpha,\;\;k=0,1,\cdots,\eta-2,
\end{equation*}
Here, $\phi_k(x)$ stands for any polynomial of degree $k$,  which
performs very similarly as the
``test function" used in general finite element analysis; \\
\textbf{Step 3.} Write down $B_\tau,\,\bar{B}_\tau$ and
$\bar{A}_{\tau,\,\sigma}$ (satisfy $\mathcal{B}(\xi)$ and
$\mathcal{CN}(\eta)$ automatically), which results in a symplectic
csRKN method of order at least
$\min\{\xi,\,2\eta+2,\,\eta+\zeta\}=\min\{\xi,\,\eta+\zeta\}$ with
$\zeta=\min\{\xi,\,\eta\}$ by Theorem \ref{ord_csRKN} and
\ref{sym_design}. If needed, we then acquire symplectic RKN methods
by using quadrature rules (see Theorem \ref{sym_quad}).

The procedure above gives a general framework for deriving
symplectic integrators. In view of Theorem \ref{qua:csRKN} and
\ref{sym_design}, it is suggested to design Butcher coefficients
with low-degree $\bar{A}_{\tau,\,\sigma}$ and $B_\tau$, and $\eta$
is better to take as $\eta\approx\frac{1}{2}\xi$. Besides, for the
sake of conveniently computing those integrals of
$\mathcal{CN}(\eta)$ in the second step, the following ansatz may be
advisable (let $\rho\geq\eta$ and $\xi\geq2\eta-1$)
\begin{equation}\label{symBA}
\begin{split}
C_\tau&=\tau,\quad
B_\tau=\sum\limits_{j=0}^{\xi-1}\int_0^1P_j(x)\,\dif x P_j(\tau),
\quad\bar{B}_\tau=B_\tau(1-\tau),\\
\bar{A}_{\tau,\sigma}&=B_\sigma\Big(\alpha_{(0,0)}+\alpha_{(0,1)}P_1(\sigma)
+\alpha_{(1,0)}P_1(\tau)+\sum_{1<i+j\in \mathbb{Z}\atop
i\leq\rho,\,j\leq \xi-\eta+1}\alpha_{(i,j)}
P_i(\tau)P_j(\sigma)\Big),
\end{split}
\end{equation}
where $\alpha_{(0,1)}-\alpha_{(1,0)}=-\big<x,\,P_1(x)\big>_w,
\;\alpha_{(i,j)}=\alpha_{(j,i)},\;i+j>1$. Because of the index $j$
restricted by $j\leq \xi-\eta+1$ in \eqref{symBA}, we can use
$\mathcal{B}(\xi)$ to arrive at (please c.f. \eqref{eq:cd1})
\begin{equation*}
\begin{split}
&\int_I\bar{A}_{\tau,\,\sigma}w(\sigma)\phi_k(\sigma)\,\dif
\sigma\\
&=\int_IB_\sigma\Big(\alpha_{(0,0)}+\alpha_{(0,1)}P_1(\sigma)
+\alpha_{(1,0)}P_1(\tau)+\sum_{1<i+j\in \mathbb{Z}\atop
i\leq\rho,\,j\leq \xi-\eta+1}\alpha_{(i,j)}
P_i(\tau)P_j(\sigma)\Big)w(\sigma)\phi_k(\sigma)\,\dif \sigma\\
&=\big(\alpha_{(0,0)}+\alpha_{(1,0)}P_1(\tau)\big)\int_0^1
\phi_k(x)\,\dif x+\alpha_{(0,1)}\int_0^1 P_1(x)\phi_k(x)\,\dif
x\\
&\;\;\;+\sum_{1<i+j\in \mathbb{Z}\atop i\leq\rho,\,j\leq
\xi-\eta+1}\alpha_{(i,j)} P_i(\tau)\int_0^1P_j(x)\phi_k(x)\,\dif
x,\quad0\leq k\leq\eta-2.
\end{split}
\end{equation*}
Therefore, $\mathcal{CN}(\eta)$ implies that
\begin{equation}\label{symp_eqs}
\begin{split}
&\big(\alpha_{(0,0)}+\alpha_{(1,0)}P_1(\tau)\big)\int_0^1
\phi_k(x)\,\dif x+\alpha_{(0,1)}\int_0^1 P_1(x)\phi_k(x)\,\dif
x\\
&+\sum_{1<i+j\in \mathbb{Z}\atop i\leq\rho,\,j\leq
\xi-\eta+1}\alpha_{(i,j)}P_i(\tau)\int_0^1P_j(x)\phi_k(x)\,\dif
x=\int_0^\tau \int_0^\alpha\phi_k(x)\,\dif x\,\dif\alpha,\quad0\leq
k\leq\eta-2.
\end{split}
\end{equation}
where $\alpha_{(0,1)}-\alpha_{(1,0)}=-\big<x,\,P_1(x)\big>_w,
\;\alpha_{(i,j)}=\alpha_{(j,i)},\;i+j>1$. Finally, it needs to
settle $\alpha_{(i,j)}$ by transposing, comparing or merging similar
items of \eqref{symp_eqs} after the polynomial on right-hand side of
\eqref{symp_eqs} being represented by the basis
$\{P_j(\tau)\}_{j=0}^\infty$. In view of the symmetry of
$\alpha_{(i,j)}$, if we let $r=\min\{\rho,\xi-\eta+1\}$, then
actually the number of degrees of freedom of these parameters is
$(r+1)(r+2)/2$, by noticing that
\begin{equation*}
\alpha_{(i,j)}=0, \;\;\text{for}\;i>r\;\text{or}\;j>r.
\end{equation*}

\subsection{Symmetric integrators}

Theoretical analyses and a large number of numerical tests indicate
that symmetric integrators applied to (near-)integrable reversible
systems share similar properties to symplectic integrators applied
to (near-)integrable Hamiltonian systems: linear error growth,
near-conservation of first integrals, existence of invariant tori
\cite{hairerlw06gni}. The good long-time behavior of symmetric
integrators motivates us to find more new integrators.

\begin{defn}\cite{hairerlw06gni}
A numerical one-step method $\phi_h$ is called symmetric (or
time-reversible) if it satisfies
$$\phi^*_h=\phi_h,$$
where $\phi^*_h=\phi^{-1}_{-h}$ is referred to as the adjoint method
of $\phi_h$.
\end{defn}
\begin{rem}
Symmetry implies that the original method and the adjoint method
give identical numerical results. A well-known property of symmetric
integrators is that they possess an \emph{even order}
\cite{hairerlw06gni}. By the definition, a one-step method
$z_1=\phi_h(z_0; t_0,t_1)$ is symmetric if exchanging
$h\leftrightarrow -h$, $z_0\leftrightarrow z_1$ and
$t_0\leftrightarrow t_1$ leaves the original method unaltered.
\end{rem}

In order to derive symmetric integrators, we assume the interval $I$
to be the following two cases:
\begin{itemize}
  \item[(i)]  $I=[a,b]$ (finite interval) with $a+b=1$;
  \item[(ii)] $I=(-\infty,+\infty)$ (infinite interval).
\end{itemize}
In what follows, we first establish the adjoint method of a given
csRKN method. From (\ref{eq:csrkn1}-\ref{eq:csrkn3}), by
interchanging $t_0, q_0, q'_0, h$ with $t_1, q_1, q'_1, -h$,
respectively, we have
\begin{subequations}
    \begin{alignat}{3}
     \label{eq:rever1}
&Q_\tau=q_1 -hC_\tau q'_1 +h^2\int_I \bar{A}_{\tau, \sigma}w(\sigma)
f(t_1-C_\sigma h, Q_\sigma) \dif \sigma, \;\;\tau \in I, \\
     \label{eq:rever2}
&q_{0}=q_1- h q'_1+h^2 \int_I \bar{B}_\tau w(\tau) f(t_1-C_\tau h, Q_\tau) \dif\tau, \\
     \label{eq:rever3}
&q'_0 = q'_1-h\int_I B_\tau w(\tau) f(t_1-C_\tau h, Q_\tau)
\dif\tau.
   \end{alignat}
\end{subequations}
Note that $t_1-C_\tau h=t_0+(1-C_\tau)h$,  \eqref{eq:rever3} becomes
\begin{equation}\label{recast01}
q'_1=q'_0+h\int_I B_\tau w(\tau)f(t_0+(1-C_\tau)h, Q_\tau) \dif\tau,
\end{equation}
substituting it into \eqref{eq:rever2} then we get
\begin{equation}\label{recast02}
q_1=q_{0}+hq'_0+h^2 \int_I (B_\tau-\bar{B}_\tau) w(\tau)
f(t_0+(1-C_\tau)h, Q_\tau)\dif\tau.
\end{equation}
Next, inserting \eqref{recast01} and \eqref{recast02} into
\eqref{eq:rever1}, it follows that
\begin{equation}\label{recast03}
Q_\tau=q_0
+h(1-C_\tau)q'_0+h^2\int_I\Big(B_\sigma(1-C_\tau)-\bar{B}_\sigma+\bar{A}_{\tau,
\sigma}\Big) w(\sigma) f(t_0+(1-C_\sigma) h, Q_\sigma) \dif \sigma.
\end{equation}
By a change of variables (replacing $\tau$ and $\sigma$ with
$1-\tau$ and $1-\sigma$ respectively), \eqref{recast03},
\eqref{recast02} and \eqref{recast01} can be recast as
\begin{equation}\label{adjo}
\begin{split}
&Q^*_\tau=q_0+hC^*_\tau q'_0 +h^2\int_I\bar{A}^*_{\tau,
\sigma}w(1-\sigma)
f(t_0+C^*_\sigma h, Q^*_\sigma) \dif \sigma, \;\;\tau \in I, \\
&q_{1}=q_0+ h q'_0+h^2 \int_I \bar{B}^*_\tau w(1-\tau) f(t_0+C^*_\tau h, Q^*_\tau) \dif\tau, \\
&q'_1 = q'_0 +h\int_I B^*_\tau w(1-\tau) f(t_0+C^*_\tau h, Q^*_\tau)
\dif\tau,
\end{split}
\end{equation}
where $Q^*_\tau=Q_{1-\tau},\,\tau\in I$ and
\begin{equation}\label{adjo:coe}
\begin{split}
C^*_\tau&=1-C_{1-\tau},\\
\bar{A}^*_{\tau,\sigma}&=B_{1-\sigma}(1-C_{1-\tau})-\bar{B}_{1-\sigma}+\bar{A}_{1-\tau,
1-\sigma},\\
\bar{B}^*_\tau&=B_{1-\tau}-\bar{B}_{1-\tau},\\
B^*_\tau&=B_{1-\tau},
\end{split}
\end{equation}
for $\tau,\,\sigma\in I$. Consequently, we get the adjoint method
given by (\ref{adjo}-\ref{adjo:coe}). Hence if we require
\begin{equation*}
C_\tau=C^*_\tau,\,\bar{A}_{\tau,\sigma}w(\sigma)=\bar{A}^*_{\tau,\sigma}w(1-\sigma),
\,\bar{B}_\tau w(\tau)=\bar{B}^*_\tau w(1-\tau),\,B_\tau
w(\tau)=B^*_\tau w(1-\tau),
\end{equation*}
then the original csRKN method is symmetric. We summarize the
results above in the following theorem.
\begin{thm}\label{symm_cond_ori}
If a csRKN method (\ref{eq:csrkn1}-\ref{eq:csrkn3}) satisfies
\begin{equation}\label{sym_conds01}
\begin{split}
C_\tau&=1-C_{1-\tau},\\
\bar{A}_{\tau,\sigma}w(\sigma)&=\Big(B_{1-\sigma}(1-C_{1-\tau})-
\bar{B}_{1-\sigma}+\bar{A}_{1-\tau,1-\sigma}\Big)w(1-\sigma),\\
\bar{B}_\tau w(\tau)&=\big(B_{1-\tau}-\bar{B}_{1-\tau}\big) w(1-\tau),\\
B_\tau w(\tau)&=B_{1-\tau} w(1-\tau),
\end{split}
\end{equation}
for $\forall\,\tau,\,\sigma\in I$, then the method is symmetric.
Particularly, if the weight function $w(x)$ satisfies $w(x)\equiv
w(1-x)$, then the symmetric condition \eqref{sym_conds01} becomes
\begin{equation}\label{sym_conds02}
\begin{split}
C_\tau&=1-C_{1-\tau},\\
\bar{A}_{\tau,\sigma}&=B_{1-\sigma}(1-C_{1-\tau})-\bar{B}_{1-\sigma}+\bar{A}_{1-\tau,1-\sigma},\\
\bar{B}_\tau &=B_{1-\tau}-\bar{B}_{1-\tau},\\
B_\tau&=B_{1-\tau},
\end{split}
\end{equation}
for $\forall\,\tau,\,\sigma\in I$.
\end{thm}
\begin{thm}\label{symm_quad}
If $w(x)\equiv w(1-x)$ and the coefficients of the underlying csRKN
method (\ref{eq:csrkn1}-\ref{eq:csrkn3}) satisfying
\eqref{sym_conds02}, then the RKN method with tableau
\eqref{RKN:qua} is symmetric, provided that the weights and nodes of
the quadrature formula satisfy $b_{s+1-i}=b_i$ and $c_{s+1-i}=1-c_i$
for all $i$.
\end{thm}
\begin{proof}
Please refer to Theorem 4.1 of \cite{Tangz18sib} for a similar
proof.
\end{proof}
\begin{rem}
Theorem \ref{symm_cond_ori} and \ref{symm_quad} is also applicable
for the case of csRKN method (\ref{Heq:csrkn1}-\ref{Heq:csrkn3}).
\end{rem}

By using orthogonal polynomial expansion technique, we acquire a
useful result for designing symmetric integrators.
\begin{thm}\label{symmcon3}
Suppose that $w(x)\equiv
w(1-x),\,C_\tau=\tau,\,\bar{B}_\tau=B_\tau(1-C_\tau)$ and
$\bar{A}_{\tau,\sigma}/B_\sigma\in L_w^2(I\times I)$, then the
symmetric condition \eqref{sym_conds01} is equivalent to the fact
that $\bar{A}_{\tau,\sigma}$ has the following form in terms of the
orthogonal polynomials $P_n(x)$ in $L_w^2(I)$
\begin{equation}\label{symmconvari}
\bar{A}_{\tau,\sigma}=B_\sigma\Big(\alpha_{(0,0)}+\alpha_{(0,1)}P_1(\sigma)
+\alpha_{(1,0)}P_1(\tau)+\sum\limits_{i+j\,\text{is}\,\text{even}
\atop 1<i+j\in \mathbb{Z}}\alpha_{(i,j)}
P_i(\tau)P_j(\sigma)\Big),\quad\alpha_{(i,j)}\in \mathbb{R},
\end{equation}
with $B_{\sigma}\equiv B_{1-\sigma}$, where
$\alpha_{(0,1)}=-\alpha_{(1,0)}=-\frac{1}{2}\big<x,\,P_1(x)\big>_w$,
provided that the orthogonal polynomials $P_n(x)$ satisfy
\begin{equation}\label{symm_relation}
P_n(1-x)=(-1)^nP_n(x),\;n\in \mathbb{Z}.
\end{equation}
\end{thm}
\begin{proof}
We only give the proof for the necessity, seeing that the
sufficiency part is rather trivial. Since under the assumption
$w(x)\equiv w(1-x)$, we have get \eqref{sym_conds02}. Hence, by
using $C_\tau=\tau,\,\bar{B}_\tau=B_\tau(1-C_\tau)$ and
$B_{\sigma}\equiv B_{1-\sigma}$, the second formula of
\eqref{sym_conds02} becomes
\begin{equation*}
\bar{A}_{\tau,\sigma}-\bar{A}_{1-\tau,1-\sigma}=B_{\sigma}(\tau-\sigma),
\end{equation*}
which leads to
\begin{equation}\label{ABsymm}
\frac{\bar{A}_{\tau,\sigma}}{B_{\sigma}}-\frac{\bar{A}_{1-\tau,1-\sigma}}{B_{1-\sigma}}=\tau-\sigma.
\end{equation}
Analogously to the proof of Theorem \ref{constr_symcsRKN}, let us
consider the expansion of $\bar{A}_{\tau,\,\sigma}/B_\sigma$ given
by \eqref{AdivB}. By using \eqref{AdivB} and \eqref{symm_relation},
it yields
\begin{equation*}
\bar{A}_{1-\tau,\,1-\sigma}/B_{1-\sigma}=\alpha_{(0,0)}-\alpha_{(0,1)}P_1(\sigma)
-\alpha_{(1,0)}P_1(\tau)+\sum\limits_{i+j>1}(-1)^{i+j}\alpha_{(i,j)}
P_i(\tau)P_j(\sigma).
\end{equation*}
By substituting \eqref{AdivB} and the above formula into
\eqref{ABsymm} and comparing the like basis, it gives
\eqref{symmconvari}.
\end{proof}

For the sake of employing Theorem \ref{symmcon3}, we also need some
useful results which are quoted from \cite{Tang18aef}.
\begin{thm}\cite{Tang18aef}
If $w(x)$ is an even function, i.e., satisfying $w(-x)\equiv w(x)$,
then the shifted function defined by $\widehat{w}(x)=w(2\theta
x-\theta)$ satisfies the symmetry relation: $\widehat{w}(x)\equiv
\widehat{w}(1-x)$. Here $\theta$ is a non-zero constant.
\end{thm}
\begin{thm}\cite{Tang18aef}\label{shiftpol}
If a sequence of polynomials $\{P_n(x)\}_{n=0}^\infty$ satisfy the
symmetry relation
\begin{equation}\label{symm_relat}
P_n(-x)=(-1)^nP_n(x),\;n\in \mathbb{Z},
\end{equation}
then the shifted polynomials defined by
$\widehat{P}_n(x)=P_n(2\theta x-\theta)$ are bound to satisfy the
property \eqref{symm_relation}. Here $\theta$ is a non-zero
constant.
\end{thm}
\begin{thm}\cite{Tang18aef}
If a sequence of polynomials $\{P_n(x)\}_{n=0}^\infty$ satisfy
\eqref{symm_relation}, then we have
\begin{equation*}
\int_0^1P_j(x)\,\dif
x=0,\;\;\text{for}\;\;j\;\;\text{is}\;\;\text{odd},
\end{equation*}
and the following function
\begin{equation}
B_\sigma=\sum\limits_{j=0}^{\xi-1}\int_0^1P_j(x)\,\dif x
P_j(\sigma)+\sum\limits_{\text{even}\;j\geq \xi}\lambda_j
P_j(\sigma),\;\; \xi\geq1,
\end{equation}
always satisfies $B_{\sigma}\equiv B_{1-\sigma}$.
\end{thm}

As pointed out in \cite{Tang18aef}, many classical (standard)
orthogonal polynomials including Hermite polynomials, Legendre
polynomials, Chebyshev polynomials of the first and second kind, and
any other general Gegenbauer polynomials etc., do not satisfy
\eqref{symm_relation}, but they possess the symmetry property
\eqref{symm_relat}. Nevertheless, by using Theorem \ref{shiftpol} we
can always shift them to a suitable interval such that the condition
\eqref{symm_relation} is fulfilled. With these discussions, we can
also propose an operational \emph{procedure} for constructing
symmetric integrators in a similar way as that given in the
preceding subsection.

\subsection{Some examples}

In the following, we provide some examples for illustrating the
application of our theoretical results. On account of
\eqref{coef_Ats}, we only present the values of $\alpha_{(i,j)}$
with $i\leq j$ in our examples. Besides, the Gaussian-Christoffel's
quadrature rules (please see \eqref{wquad}) will be used, which
means the quadrature nodes $c_1, c_2,\cdots, c_s$ are exactly the
zeros of the normalized orthogonal polynomial $P_s(x)$ in
$L^2_w(I)$. For the sake of deriving symmetric methods we mainly
consider the weighted orthogonal polynomials shifted into a suitable
interval.

\begin{exa}\label{Legend}
Consider using the shifted normalized Legendre polynomials which are
orthogonal with respect to the weight function $w(x)=1$ on $[0,1]$.
These Legendre polynomials $L_n(x)$ can be defined by Rodrigues'
formula \cite{hairerw96sod}
\begin{equation*}
L_0(x)=1,\;L_n(x)=\frac{\sqrt{2n+1}}{n!}\frac{{\dif}^n}{\dif t^n}
\Big(t^n(t-1)^n\Big),\;x\in[0,1],\;\;n=1,2,\cdots.
\end{equation*}
\end{exa}
Let $\xi=3,\,\eta=2,\,\rho=2$ in \eqref{symBA}, $r=2$ and thus the
number of degrees of freedom is $(r+1)(r+2)/2=6$. For simplicity, we
set $\alpha_{(i,j)}=0$ for $0\leq i, j\leq2,i+j>2$. After some
elementary calculations, it gives
\begin{equation*}
\alpha_{(0,0)}=\frac{1}{6},\;\;\alpha_{(0,1)}=-\frac{\sqrt{3}}{12},\;
\;\alpha_{(1,1)}=\mu,\;\;\alpha_{(0,2)}=\frac{\sqrt{5}}{60},
\end{equation*}
where $\alpha_{(1,1)}=\mu$ is a free parameter, then we get a
$\mu$-parameter family of symmetric (by Theorem \eqref{symmcon3})
and symplectic csRKN methods of order $4$. By using
Gauss-Christoffel's quadrature rules with $2$ nodes, we get a family
of symmetric and symplectic RKN methods of order $4$ which are shown
in Tab. \ref{exa1:symp01} with $\gamma=\frac{1}{2}\mu$. It is found
that this family of methods coincides with the methods presented in
\cite{Tangsz18hos}.

\begin{table}
\[\ba{c|cc} \frac{3-\sqrt{3}}{6}&
\frac{1}{12}+\gamma&
\frac{1-\sqrt{3}}{12}-\gamma\\[2pt]
\frac{3+\sqrt{3}}{6} & \frac{1+\sqrt{3}}{12}-\gamma&
\frac{1}{12}+\gamma\\[2pt]
\hline & \frac{3+\sqrt{3}}{12}& \frac{3-\sqrt{3}}{12}\\[2pt]
\hline & \frac{1}{2}& \frac{1}{2} \ea \] \caption{A family of
$2$-stage $4$-order symmetric and symplectic RKN methods, based on
the shifted Legendre polynomials $L_n(x)$.}\label{exa1:symp01}
\end{table}

\begin{exa}\label{Cheby}
Consider using the shifted normalized Chebyshev polynomials of the
first kind which are orthogonal with respect to the weight function
$w(x)=\frac{1}{2\sqrt{x(1-x)}}$ on $[0,1]$. These Chebyshev
polynomials $T_n(x)$ can be defined by \cite{Tang18csm}
\begin{equation*}
T_0(x)=\frac{\sqrt{2}}{\sqrt{\pi}},\;\;T_n(x)=\frac{2\cos
\big(n\arccos(2x-1)\big)}{\sqrt{\pi}},\;x\in[0,1],\;\;n=1,2,\cdots.
\end{equation*}
\end{exa}
Let $\xi=3,\,\eta=2,\,\rho=2$ in \eqref{symBA} and set
$\alpha_{(i,j)}=0$ for $0\leq i, j\leq2,i+j>2$. After some
elementary calculations, it gives
\begin{equation*}
\alpha_{(0,0)}=\frac{5}{24},\;\;\alpha_{(0,1)}=-\frac{\sqrt{\pi}}{8},\;
\;\alpha_{(1,1)}=\mu,\;\;\alpha_{(0,2)}=\frac{\sqrt{2}}{64}\pi,
\end{equation*}
where $\alpha_{(1,1)}=\mu$ is a free parameter, then we get a
$\mu$-parameter family of symplectic csRKN methods of order at least
$3$. However, since we have
$\alpha_{(0,1)}=-\frac{1}{2}\big<x,\,T_1(x)\big>_w=-\frac{\sqrt{\pi}}{8}$
and other conditions of Theorem \ref{symmcon3} are fulfilled, the
newly-derived methods are symmetric and of an even order $4$. By
using Gauss-Christoffel's quadrature rules with $3$ nodes, we get a
family of symmetric and symplectic RKN methods of order $4$ which
are shown in Tab. \ref{exa2:symp01} with $\gamma=\frac{2\mu}{3\pi}$.

\begin{table}
\[\ba{c|ccc} \frac{2-\sqrt{3}}{4} & \frac{13}{216}+\gamma & \frac{85-60\sqrt{3}}{864} &
\frac{13-12\sqrt{3}}{216}-\gamma\\[2pt]
\frac{1}{2} &\frac{17+12\sqrt{3}}{432} & \frac{5}{108} &
\frac{17-12\sqrt{3}}{432}\\[2pt]
\frac{2+\sqrt{3}}{4} & \frac{13+12\sqrt{3}}{216}-\gamma &
\frac{85-60\sqrt{3}}{864} & \frac{13}{216}+\gamma\\[2pt]
\hline & \frac{2+\sqrt{3}}{18}  & \frac{5}{18} & \frac{2-\sqrt{3}}{18} \\[2pt]
\hline & \frac{2}{9}  & \frac{5}{9} & \frac{2}{9} \ea \]  \caption{A
family of $3$-stage $4$-order symmetric and symplectic RKN methods,
based on the shifted Chebyshev polynomials of the first kind
$T_n(x)$.}\label{exa2:symp01}
\end{table}

\begin{table}
\[\ba{c|ccc} \frac{2-\sqrt{6}}{4} & \frac{11}{216}+\gamma &
\frac{91-42\sqrt{6}}{432} & \frac{11-6\sqrt{6}}{216}-\gamma\\[2pt]
\frac{1}{2} & \frac{13+6\sqrt{6}}{432} &  \frac{7}{108} & \frac{13-6\sqrt{6}}{432}\\[2pt]
\frac{2+\sqrt{6}}{4} & \frac{11+6\sqrt{6}}{216}-\gamma &
\frac{91+42\sqrt{6}}{432} &
 \frac{11}{216}+\gamma\\[2pt]
\hline & \frac{2+\sqrt{6}}{36} & \frac{7}{18} & \frac{2-\sqrt{6}}{36}\\[2pt]
\hline & \frac{1}{9} & \frac{7}{9} & \frac{1}{9}\ea\] \caption{A
family of $3$-stage $4$-order symmetric and symplectic RKN methods,
based on the shifted Hermite polynomials
$\widehat{H}_n(x)$.}\label{exa3:symp01}
\end{table}

\begin{table}
\[\ba{c|ccc} -\frac{\sqrt{6}}{2} & \frac{4-3\sqrt{6}}{432}
& \frac{70-21\sqrt{6}}{108} &  \frac{40+87\sqrt{6}}{432}\\[2pt]
0 & \frac{-7-9\sqrt{6}}{216} &  \frac{7}{108} & \frac{-7+9\sqrt{6}}{216}\\[2pt]
\frac{\sqrt{6}}{2} & \frac{40-87\sqrt{6}}{432} &
\frac{70+21\sqrt{6}}{108} & \frac{4+3\sqrt{6}}{432} \\[2pt]
\hline & \frac{-5-\sqrt{6}}{36} & \frac{7}{9} & \frac{-5+\sqrt{6}}{36}\\[2pt]
\hline & \frac{4-3\sqrt{6}}{36} & \frac{7}{9} &
\frac{4+3\sqrt{6}}{36} \ea\] \caption{A $3$-stage $3$-order
symplectic RKN method (non-symmetric), based on the Hermite
polynomials $H_n(x)$.}\label{exa3:symp02}
\end{table}

\begin{exa}\label{Hermite}
Consider using the shifted normalized Hermite polynomials which are
orthogonal with respect to the weight function
$\widehat{w}(x)=e^{-(2x-1)^2}$ on $(-\infty,+\infty)$. These Hermite
polynomials $\widehat{H}_n(x)$ can be defined by \cite{Tang18aef}
\begin{equation*}
\widehat{H}_n(x)=\sqrt{2}H_n(2x-1),\;\;x\in(-\infty,+\infty),\;\;n=0,1,\cdots,
\end{equation*}
where $H_n(x)$ is the standard normalized $n$-degree Hermite
polynomial
\begin{equation}\label{standardHerm}
H_0(x)=\frac{1}{\pi^{\frac{1}{4}}},\;H_n(x)=\frac{(-1)^ne^{x^2}}{\sqrt{2^nn!}\pi^{\frac{1}{4}}}
\frac{\dif^n}{dx^n}\big(e^{-x^2}\big),\;\;x\in(-\infty,+\infty),\;\;n=1,2,\cdots,
\end{equation}
with the weight function given by $w(x)=e^{-x^2}$.
\end{exa}
Let $\xi=3,\,\eta=2,\,\rho=2$ in \eqref{symBA} and set
$\alpha_{(i,j)}=0$ for $0\leq i, j\leq2,i+j>2$. After some
elementary calculations, it gives
\begin{equation*}
\alpha_{(0,0)}=\frac{5}{24},\;\;\alpha_{(0,1)}=-\frac{\pi^{\frac{1}{4}}}{8},\;
\;\alpha_{(1,1)}=\mu,\;\;\alpha_{(0,2)}=\frac{\sqrt{2\pi}}{32},
\end{equation*}
where $\alpha_{(1,1)}=\mu$ is a free parameter, then we get a
$\mu$-parameter family of symplectic csRKN methods of order at least
$3$. However, since we have
$\alpha_{(0,1)}=-\frac{1}{2}\big<x,\,\widehat{H}_1(x)\big>_{\widehat{w}}
=-\frac{\pi^{\frac{1}{4}}}{8}$ and other conditions of Theorem
\ref{symmcon3} are fulfilled, the newly-derived methods are
symmetric and of an even order $4$. By using Gauss-Christoffel's
quadrature rules with $3$ nodes, we get a family of symmetric and
symplectic RKN methods of order $4$ which are shown in Tab.
\ref{exa3:symp01} with $\gamma=\frac{2\mu}{3\sqrt{\pi}}$.

We claim that if we do not use the shifted Hermite polynomials
$\widehat{H}_n(x)$, then it may result in symplectic methods without
the symmetric property. Let us consider using $H_n(x)$ to construct
symplectic methods, take the same $\xi=3,\,\eta=2,\,\rho=2$, and
also set $\alpha_{(i,j)}=0$ for $0\leq i, j\leq2,i+j>2$.
Additionally, we impose
$\alpha_{(0,1)}=-\frac{1}{2}\big<x,\,H_1(x)\big>_w
=-\frac{\sqrt{2}\pi^{\frac{1}{4}}}{4}$, then we get
\begin{equation*}
\alpha_{(0,0)}=\frac{7}{12},\;\;\alpha_{(0,1)}=-\frac{\sqrt{2}\pi^{\frac{1}{4}}}{4},\;
\;\alpha_{(1,1)}=-\frac{\sqrt{\pi}}{2},\;\;\alpha_{(0,2)}=\frac{\sqrt{2\pi}}{4}.
\end{equation*}
In such a case, we get a symplectic csRKN method with order $3$. One
can verify that such method does not satisfy all the bushy tree
order condition for order $4$ (see $\mathcal{B}(\xi)$), e.g.,
\begin{equation*}
\int_{-\infty}^{+\infty}B_\tau w(\tau) C_\tau^3\,\dif
\tau=\int_{-\infty}^{+\infty}\frac{7+6\tau-2\tau^2}{6\sqrt{\pi}}e^{-\tau^2}\tau^3\,\dif
\tau=\frac{3}{4}\neq\frac{1}{4},
\end{equation*}
hence it is of an odd order and can not be a symmetric method.
Besides, by using corresponding Gauss-Christoffel's quadrature rules
with $3$ nodes, we get a $3$-stage $3$-order symplectic RKN method
which is shown in Tab. \ref{exa3:symp02}. Although it looks like as
if the quadrature weights and nodes possess a kind of ``symmetry",
the method is essentially not symmetric according to the classical
symmetric conditions for RKN methods \cite{okunbors92ecm}.

\section{Numerical tests}\label{sec:numerical_examples}

In this section, we perform some numerical tests to verify our
theoretical results. For ease of description and comparison studies,
we denote our methods shown in
Tab.~\ref{exa1:symp01}-\ref{exa3:symp02} (with $\gamma=0$) by
Legendre-4, Chebyshev-4, Hermite-4 and Hermite-3 in turn and all of
them will be applied to two classical mechanical problems.

\begin{figure}
\begin{center}
\scalebox{0.9}[0.60]{\includegraphics{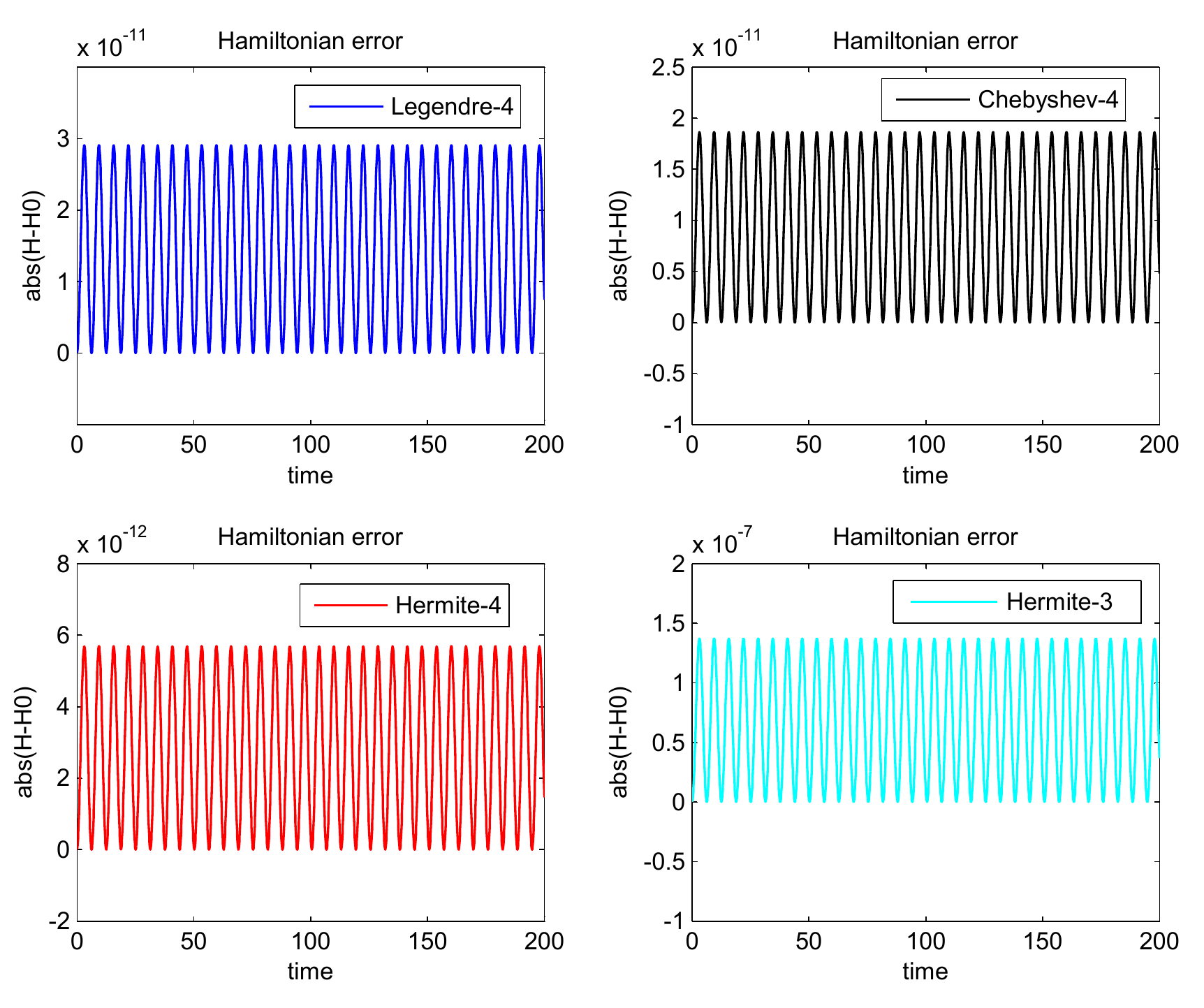}}
\caption{Energy (Hamiltonian) errors by four new symplectic RKN
methods for Kepler's problem, with step size $h=0.1$.}\label{ex1_f1}
\scalebox{0.9}[0.60]{\includegraphics{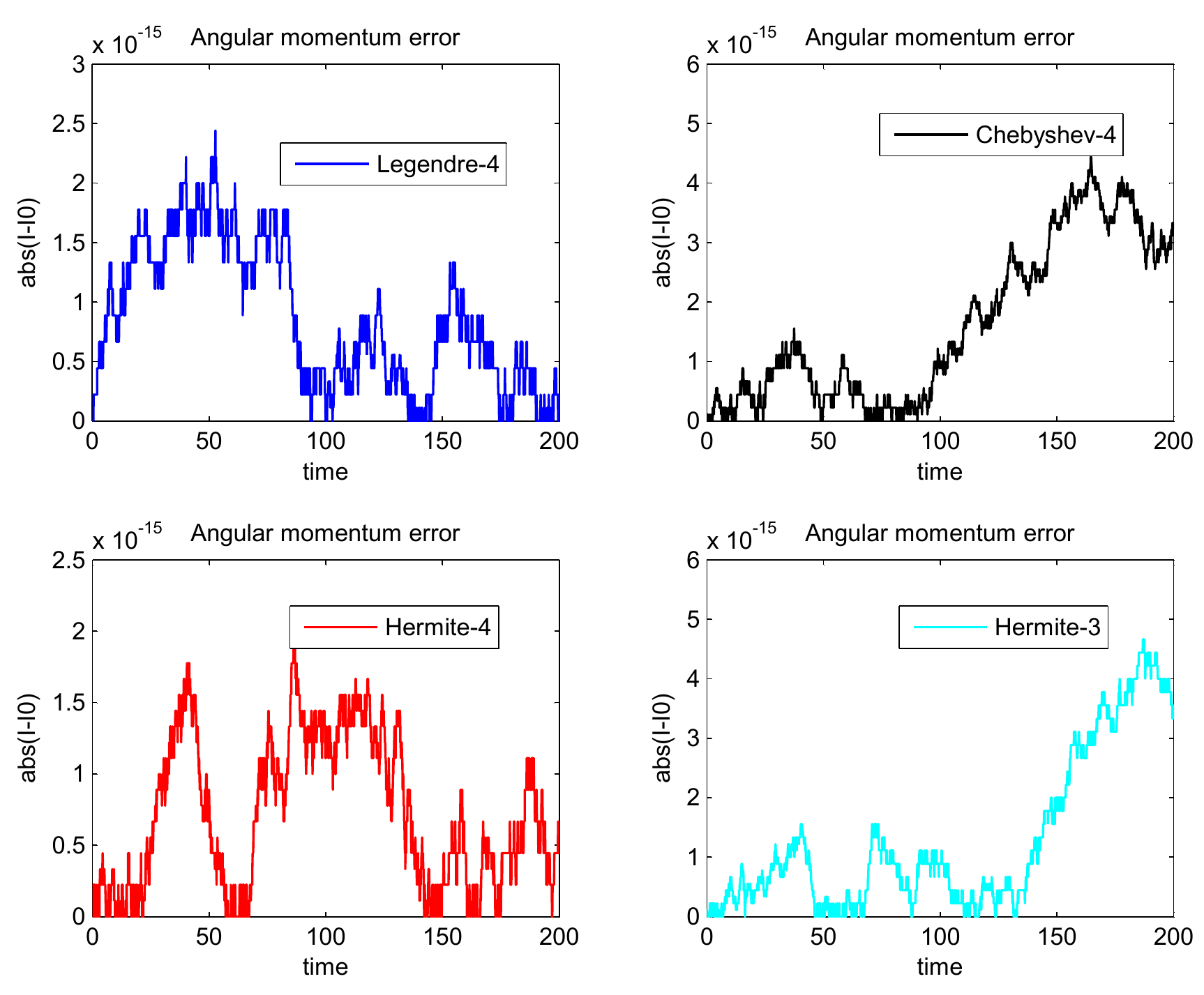}}
\caption{Angular momentum errors by four new symplectic RKN methods
for Kepler's problem, with step size $h=0.1$.}\label{ex1_f2}
\end{center}
\end{figure}

\begin{figure}
\begin{center}
\scalebox{0.9}[0.60]{\includegraphics{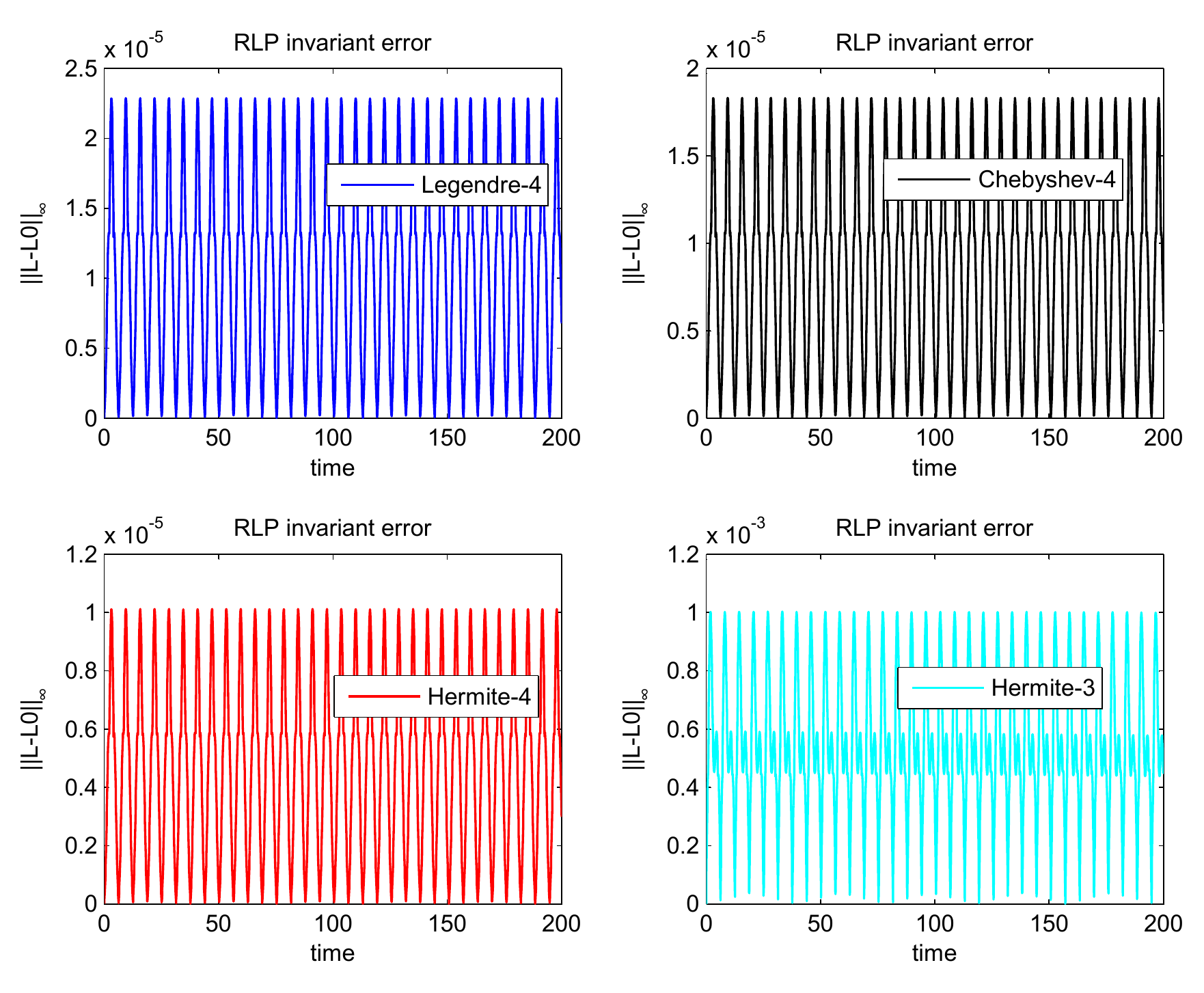}} \caption{RLP
invariant errors by four new symplectic RKN methods for Kepler's
problem, with step size $h=0.1$.}\label{ex1_f3}
\scalebox{0.9}[0.60]{\includegraphics{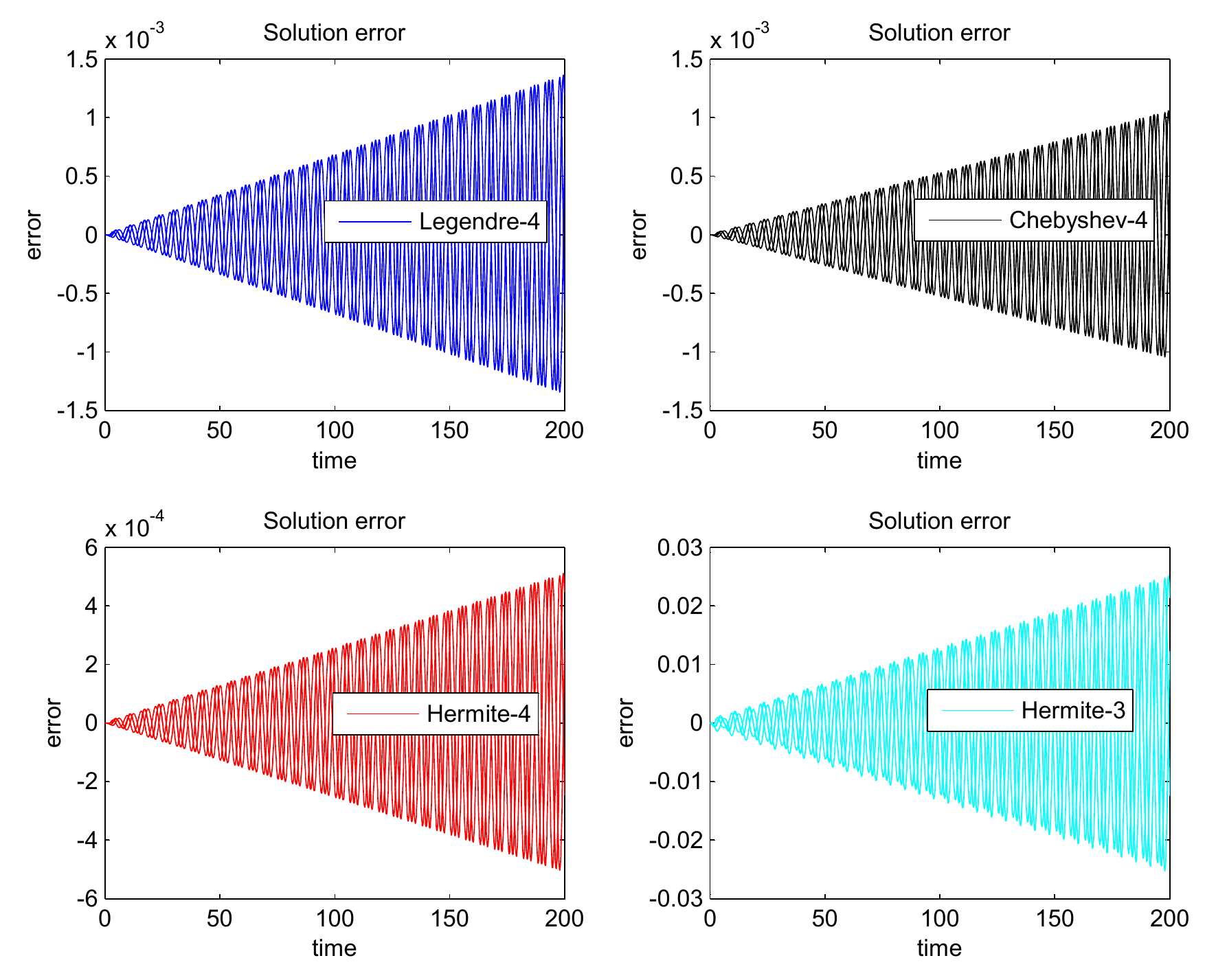}}
\caption{Solution errors by four new symplectic RKN methods for
Kepler's problem, with step size $h=0.1$.}\label{ex1_f4}
\end{center}
\end{figure}

\begin{figure}
\begin{center}
\scalebox{0.9}[0.60]{\includegraphics{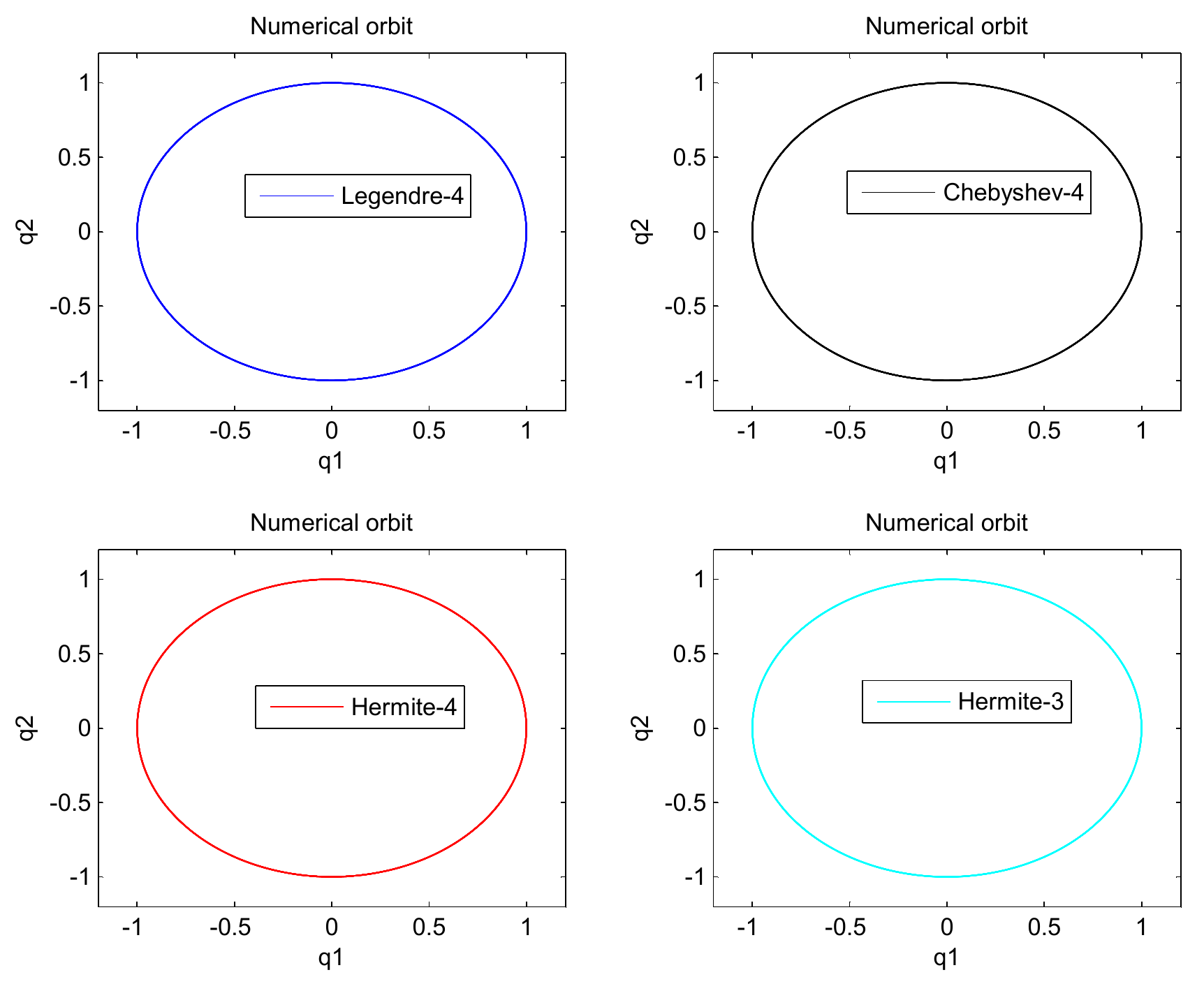}}
\caption{Numerical orbits by four new symplectic RKN methods for
Kepler's problem, with step size $h=0.1$.}\label{ex1_f5}
\end{center}
\end{figure}

\begin{exa}
Consider the numerical integration of the well-known Kepler's
problem \cite{hairerlw06gni}. The Kepler's problem describes the
motion of two bodies which attract each other under the universal
gravity. The motion of two-bodies can be described by
\begin{equation}\label{Kepler}
q''_1=-\frac{q_1}{(q_1^2+q_2^2)^{\frac{3}{2}}}, \quad
q''_2=-\frac{q_2}{(q_1^2+q_2^2)^{\frac{3}{2}}}.
\end{equation}
\end{exa}
By introducing the momenta $p_1=q'_1, p_2=q'_2$, we can transform
\eqref{Kepler} into a nonlinear Hamiltonian system with Hamiltonian
\begin{equation*}
H=\frac{1}{2}(p_1^2+p_2^2)-\frac{1}{\sqrt{q_1^2+q_2^2}}.
\end{equation*}
Beside the Hamiltonian, the system possesses other two invariants:
the quadratic angular momentum
\begin{equation*}
I=q_1p_2-q_2p_1=q^T\left(
\begin{array}{cc}
0 & 1 \\ -1 & 0 \\
\end{array}
\right)q',\;\;q=\left(
\begin{array}{c}
q_1 \\
q_2 \\
\end{array}
\right),
\end{equation*}
and the Runge-Lenz-Pauli-vector (RLP) invariant
\begin{equation*}
L=\left(
    \begin{array}{c}
      p_1 \\
      p_2 \\
      0 \\
    \end{array}
  \right)\times
  \left(
  \begin{array}{c}
  0 \\
  0 \\
  q_1p_2-q_2p_1 \\
  \end{array}
  \right)-\frac{1}{\sqrt{q_1^2+q_2^2}}\left(
  \begin{array}{c}
  q_1 \\
  q_2 \\
  0 \\
  \end{array}
  \right).
\end{equation*}
In our numerical tests, we take the initial values as
\begin{equation*}
q_1(0)=1, \;q_2(0)=0,\;p_1(0)=0, \;p_2(0)=1,
\end{equation*}
and the corresponding exact solution is
\begin{equation*}
q_1(t)=\cos(t),\;\;q_2(t)=\sin(t),\;\;p_1(t)=-\sin(t),\;\;p_2(t)=\cos(t).
\end{equation*}

Applying our symplectic integrators to \eqref{Kepler}, we compute
the approximation errors of the numerical solution to the exact
solution, as well as the errors in terms of the above three
invariants. These errors are shown in Fig.
\ref{ex1_f1}-\ref{ex1_f4}, where the errors at each time step are
carried out in the maximum norm
$||x||_{\infty}=\max(|x_1|,\cdots,|x_n|)$ for $x=(x_1,\cdots,x_n)\in
\mathbb{R}^n$. It indicates that all the symplectic integrators show
a near-preservation of the Hamiltonian and RLP invariant, and a
practical preservation (up to the machine precision) of the
quadratic angular momentum --- symplectic RKN methods can preserve
all quadratic invariants of the form $q^TDq'$ with $D$ a
skew-symmetric matrix (see \cite{hairerlw06gni}, page 104). The
solution errors of $p$-variable and $q$-variable measured in
Euclidean norm are shown in Fig. \ref{ex1_f4} which implies a linear
error growth. It is observed that amongst four methods the Hermite-4
method gives the best result, while the Hermite-3 method is inferior
to other three methods due to its lower accuracy. Moreover, all the
numerical orbits by four methods (see Fig. \ref{ex1_f5}) are in the
shape of an ellipse, closely approximating to the exact one (we do
not show it here). These numerical observations have well conformed
with the common features of symplectic integration.

\begin{figure}
\begin{center}
\scalebox{0.9}[0.60]{\includegraphics{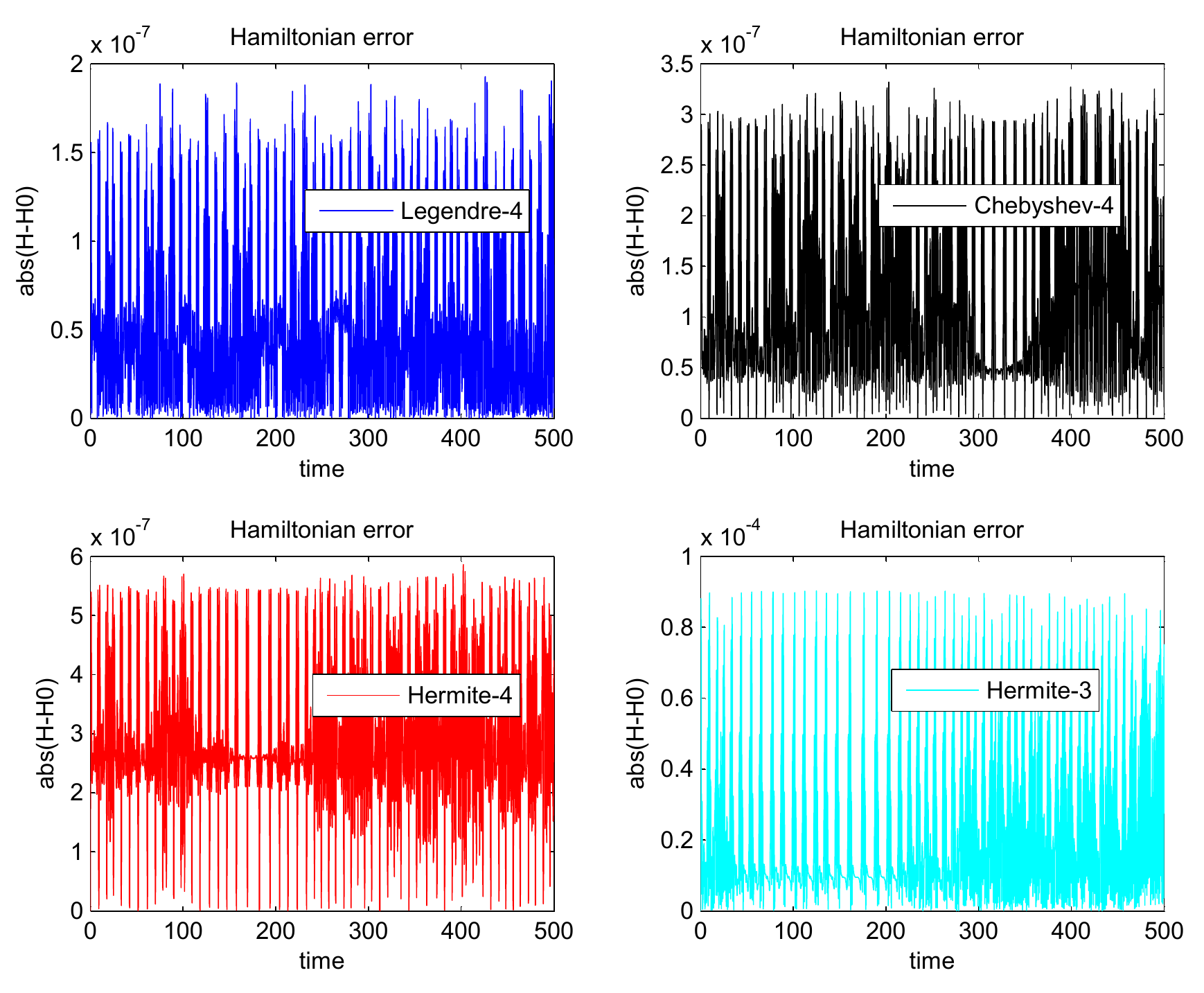}}
\caption{Energy (Hamiltonian) errors by four new symplectic RKN
methods for H\'{e}non-Heiles model problem, with step size
$h=0.1$.}\label{ex2_f1}
\scalebox{0.9}[0.60]{\includegraphics{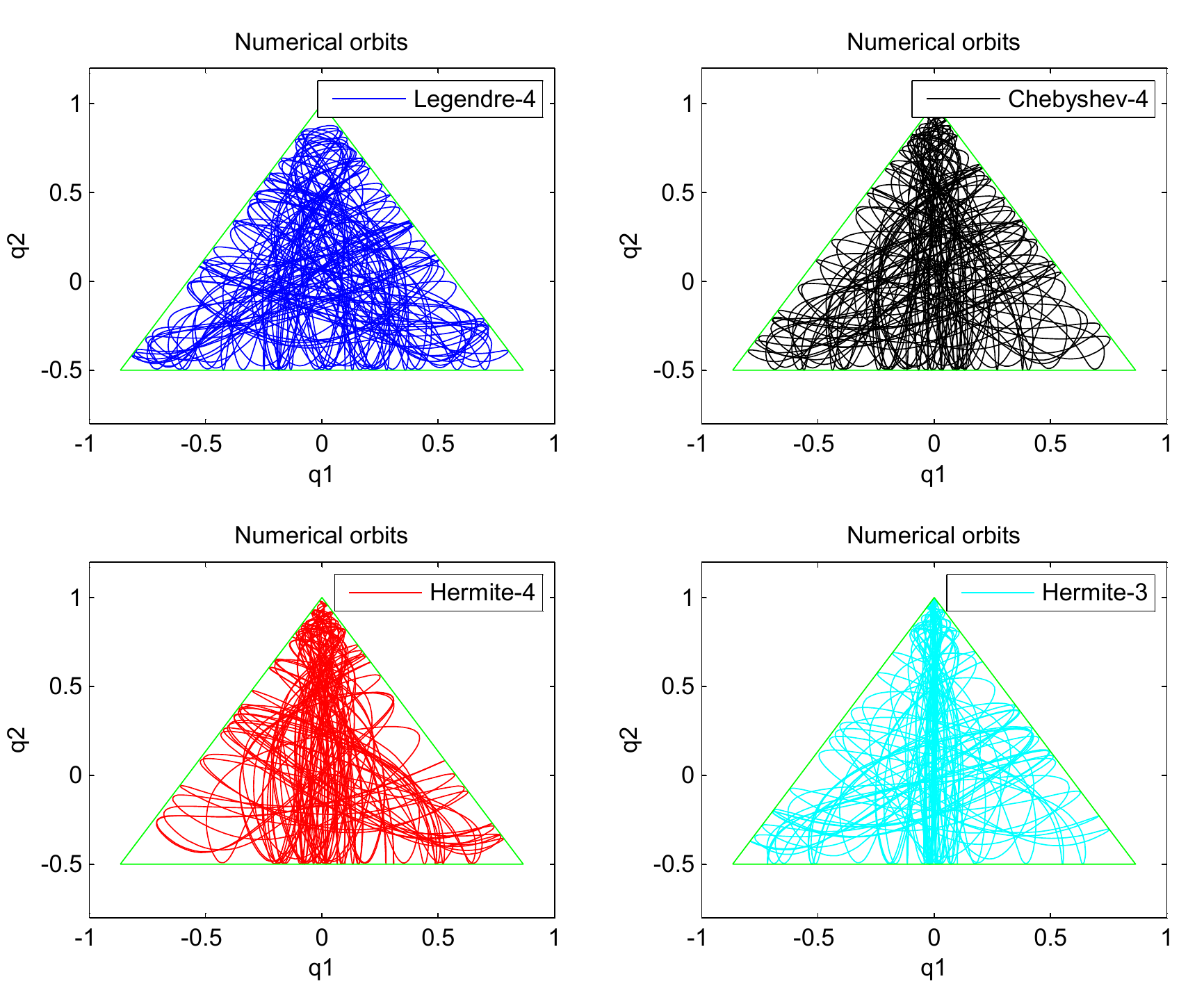}}
\caption{Chaotic orbits by four new symplectic RKN methods for
H\'{e}non-Heiles model problem, with step size
$h=0.1$.}\label{ex2_f2}
\end{center}
\end{figure}

\begin{exa}
Consider the numerical integration of the well-known
H\'{e}non-Heiles model problem \cite{hairerlw06gni}, which was
created for describing stellar motion. The problem can be described
by
\begin{equation}\label{HH}
q''_1=-q_1-2q_1q_2, \quad q''_2=-q_2-q_1^2+q_2^2.
\end{equation}
\end{exa}
It is clear to see that \eqref{HH} can be reduced to a first-order
Hamiltonian system determined by the Hamiltonian
\begin{equation*}
H=\frac{1}{2}(p_1^2+p_2^2)+\frac{1}{2}(q_1^2+q_2^2)+q_1^2q_2-\frac{1}{3}q_2^3.
\end{equation*}
In our experiment, the initial values are taken as
\begin{equation*}
q_1(0)=0.1, \;q_2(0)=-0.5,\;p_1(0)=0, \;p_2(0)=0,
\end{equation*}
which will result in a chaotic behavior and the chaotic orbits
should stay in the interior zone of an equilateral triangle
\cite{hairerlw06gni,quispelm08anc}. We present our numerical results
in Fig. \ref{ex2_f1} and \ref{ex2_f2}. It is observed that all the
symplectic methods have a well near-preservation of the energy (see
Fig. \ref{ex2_f1}) and they numerically reproduce the correct
behavior of the original system without points escaping from the
equilateral triangle (see Fig. \ref{ex2_f2}).

\section{Concluding remarks}

The constructive theory of continuous-stage Runge-Kutta-Nystr\"{o}m
methods is examined in this paper. We establish a new framework for
such methods by leading weight function into the formalism and
imposing the range of integration to be a general interval $I$
(finite or infinite). Particularly, we intensively discuss its
applications in the geometric integration of second-order
differential equations. A systematic way for deriving symplectic and
symmetric integrators is presented. We stress that our crucial
technique for deriving these geometric integrators is the orthogonal
polynomial expansion and the simplifying assumptions for order
conditions. It is hoped that in the forthcoming future other new
applications of the presented theoretical results will be
discovered.

\section*{Acknowledgements}

The author was supported by the National Natural Science Foundation
of China (11401055), China Scholarship Council and Scientific
Research Fund of Hunan Provincial Education Department (15C0028).

\end{document}